\documentclass[12pt, reqno]{amsart}
\usepackage{amsmath,amssymb,amsthm}
\usepackage{amsmath, amssymb}
\usepackage{amsthm, amsfonts, mathrsfs}
\usepackage{mathptmx}
\usepackage{fullpage}
\usepackage{amsfonts,graphicx}
\numberwithin{equation}{section}
\usepackage[colorlinks=true, pdfstartview=FitV, linkcolor=blue, citecolor=blue, urlcolor=blue]{hyperref}

\numberwithin{equation}{section}

\numberwithin{equation}{section}
\newtheorem{theorem}{Theorem}[section]

\newtheorem{lemma}[theorem]{Lemma}

\numberwithin{equation}{section}

\theoremstyle{remark}

\def\T{ \mathbb{T} }
\newcommand{\q}{{\mathbb{Q}}}
\newcommand{\p}{{\mathbb{P}}}
\newcommand{\R}{{\mathbb{R}}}
\def\div{ \hbox{\rm div}\,  }
\def\u{ \mathbf{u} }
\def\G{ \mathbf{G} }
\def\n{ \mathbf{n} }
\def\b{ \mathbf{B} }
\def\T{ \mathbb{T} }
\newcommand{\Z}{{\mathbb{Z}}}
\def\la{ \Lambda }
\def\nn{\nonumber}
\newcommand{\norm}[2]{\left\lVert #1 \right\rVert_{#2}}

\begin{document}

\title[3D Compressible MHD equations]{Global small solutions to the 3D compressible viscous non-resistive MHD system}
\author[J. Wu]{Jiahong Wu}
\address[J. Wu]{Department of Mathematics, University of Notre
	Dame, Notre Dame, IN 46556, USA } \email{jwu29@nd.edu}

 \author[X. Zhai]{Xiaoping Zhai}
\address[X. Zhai]{School of Mathematics and Statistics, Guangdong University of Technology,
Gua-ngzhou, 510520, China} \email{pingxiaozhai@163.com (Corresponding author)}

\date{}

\begin{abstract}
Whether or not smooth solutions to the 3D compressible magnetohydrodynamic (MHD) equations
without magnetic diffusion are always global in time remains an extremely challenging open problem.
No global well-posedness or stability result is currently available for this 3D MHD system
in the whole space $\mathbb R^3$ or the periodic box $\mathbb T^3$ even
when the initial data is small or near a steady-state solution. This paper presents a global existence
and stability result for smooth solutions to this 3D MHD system near any background magnetic field satisfying
a Diophantine condition.
%This result was inspired by a recent work of W. Chen, Z. Zhang and J. Zhou on
%3D incompressible MHD equations \cite{zhangzhifei}.
\end{abstract}

\maketitle

\tableofcontents

\section{ Introduction and main result}

\subsection {Model and synopsis of related studies}

We consider the 3D viscous compressible  magnetohydrodynamic (MHD) system without magnetic diffusion,
\begin{eqnarray}\label{m1}
	\left\{\begin{aligned}
		&\partial_t \rho + {\mathrm{div}} (\rho \mathbf{u}) =0, \quad t>0, \,\, x\in \mathbb \T^3, \\
		&\rho\partial_t \mathbf{u}+ \rho\mathbf{u}\cdot\nabla \mathbf{u}-\mu\Delta \mathbf{u} - (\lambda+\mu) \nabla {\mathrm{div}} \mathbf{u}+\nabla P=(\nabla\times\mathbf{B})\times\mathbf{B},\\
		&\partial_t \mathbf{B}+\mathbf{u}\cdot\nabla\mathbf{B}-\mathbf{B}\cdot\nabla\mathbf{u}
=-\mathbf{B}\div\mathbf{u},\\
		&{\mathrm{div}} \mathbf{B} =0,
	\end{aligned}\right.
\end{eqnarray}
where $\mathbb T^3$ is the 3D periodic box, and $\rho =\rho(t,x)$,  $\u= \u(t,x)$ and  $\mathbf{B}= \b(t,x)$
denote the fluid density, velocity field and the magnetic field, respectively. The parameters $\mu$
and $\lambda$ are shear viscosity and volume viscosity coefficients,  respectively,  which satisfy the standard  strong parabolicity assumption,
\begin{align*}
	\mu>0\quad\hbox{and}\quad
	\nu\stackrel{\mathrm{def}}{=}\lambda+2\mu>0.
\end{align*}
The pressure   $P=P(\rho)$ is a smooth function of the density satisfying $P'>0$ and $P'(\bar\rho)=1$ for some   constant reference density $\bar\rho>0.$ The MHD system concerned here arises in modeling fluids that can be treated as perfect conductors such as strongly collisional plasmas. In addition, the breakdown of ideal MHD is known to be the cause of solar flares, the largest explosions in the solar system \cite{Tand}.

\vskip .1in
This paper focuses on the small data global well-posedness and stability problem. This problem is
extremely challenging due to the lack of magnetic diffusion. In fact, it remains an open problem to construct global  solutions  of  the compressible non-resistive MHD system even with small initial data in $\R^3$ or $\T^3$. This work makes some progress on this difficult problem. It establishes the small
data global well-posedness and stability near background magnetic fields satisfying a Diophantine condition. Our research was inspired by a recent work of W. Chen, Z. Zhang and J. Zhou \cite{zhangzhifei}, which solved a stability problem on the 3D incompressible MHD equations.

\vskip .1in
A vector  ${\n}\in\R^3$ is said to satisfy the Diophantine condition if, for any $\mathbf{k}\in\Z^3\setminus \{0\},$
\begin{equation}\label{diufantu}
	|\n\cdot \mathbf{k}|\ge \frac{c}{|\mathbf{k}|^r} \quad\mbox{for some $c>0$ and $r>2$.}
\end{equation}
There are vectors that do not satisfy the Diophantine condition such as those with all three components
being rational. But almost all vectors in ${\mathbb R}^3$ do satisfy (\ref{diufantu}), as demonstrated
in \cite{zhangzhifei}.

\vskip .1in
We consider the evolution of perturbations near a background magnetic field $\n$
satisfying the Diophantine condition defined above. For the sake of simplicity, we still use  $\b$ to denote the perturbation $\b -{\n}$. Then the perturbed equations can be rewritten as
\begin{eqnarray}\label{m2}
	\left\{\begin{aligned}
		&\partial_t \rho + \div (\rho \u) =0,\\
		&\rho\partial_t \u+ \rho\u\cdot\nabla \u-\mu\Delta \u - (\lambda+\mu) \nabla \div \u+\nabla P={\n}\cdot\nabla \b+\b\cdot\nabla\b-\n\nabla \b-\b\nabla\b,\\
		&\partial_t \b+\u\cdot\nabla\b={\n}\cdot\nabla \u+\b\cdot\nabla\u-\n\div\u-\b\div\u,\\
		&\div \b =0,\\
		&(\rho,\u,\b)|_{t=0}=(\rho_0,\u_0,\b_0),
	\end{aligned}\right.
\end{eqnarray}
where $\n\nabla \b= \sum_{i=1}^3 \n_i \nabla \b_i$ and $\b\nabla \b= \sum_{i=1}^3 \b_i \nabla \b_i$.

\vskip .1in
Before stating our main result, we briefly summarize some related work.
When $\mathbf{B} = \mathbf{0}$,  \eqref{m1} reduces to the  isentropic  compressible Naiver-Stokes equations.  Fundamental issues on the compressible Naiver-Stokes equations such as well-posedness and blowup problems have been extensively
investigated (see, e.g.,  \cite{F04}, \cite{hoff1}, %\cite{xin3},
\cite{Xin4}, \cite{Xin}, \cite{xin7}).
When the density $\rho$ is a constant, \eqref{m1} becomes the viscous non-resistive incompressible
MHD system, which is the subject of numerous investigations (see, e.g.,  \cite{abidi}, \cite{chemin}, \cite{fefferman1}, \cite{fefferman2},  \cite{lijinlu}, \cite{LXZ}, \cite{LiZh1}, \cite{PZZu1}, \cite{RWZ}, \cite{Wa1}, \cite{wujihong1}, \cite{XZ}, \cite{ZT}). The compressible MHD equations with both
dissipation and magnetic diffusion are also the focus of many studies
(see, e.g., \cite{ChenWang}, \cite{hongguangyi}, \cite{HuWang},  \cite{lihailiang}, \cite{wuguochun}, \cite{Xiao}).
Needless to say, the references listed above is just
a very small portion of the vast literature on this subject. In contrast, not that many results are
currently available for the compressible viscous MHD equations without magnetic diffusion. The small
data global existence and stability for the 2D compressible MHD equation without magnetic diffusion
was studied by %Hu \cite{Hu},
Wu and Wu \cite{WuWu} and Wu and Zhu \cite{WuZhu}.
%Hu \cite{Hu} construction of a global solution in hybrid Besov space with an initial magnetic field near a constant background and the deformation Lagrangian matrix near the identity.
Wu and Wu \cite{WuWu} presented a
new and systematic approach on the well-posedness and stability problem on partially dissipated
systems. \cite{WuWu} introduced a diagonalization process to understand the spectrum structure and large-time behavior. The work of Wu and Zhu \cite{WuZhu} is the first to investigate such system on bounded domains and the first to solve this problem by pure energy estimates, which help reduce the complexity in other approaches.  Very recently Dong, Wu and Zhai \cite{zhaixiaoping2021arxiv}
proved the global
existence of strong solutions to the $2\frac12$-D compressible non-resistive MHD equations with small initial data.
%Tan and Wang \cite{TW} obtained the global existence of smooth solutions to the 3D compressible baroscopic viscous non-resistive MHD system in the horizontally infinite flat layer $\Omega=\R^2\times(0,1)$.
Jiang and Jiang \cite{jiangfei2019} showed the  stability/instability criteria for the stratified compressible magnetic
Rayleigh-Taylor problem in Lagrangian coordinates in the three-dimensional case. Besides these stability results, several other existence and uniqueness results are also available.   \cite{JZW} studied the non-resistive limit and the magnetic boundary layer of the 1D compressible MHD equation. \cite{LS1D} obtained
the global existence of weak solutions and their large-time behavior for this 1D MHD system.  Zhong \cite{zhongxin} obtained the local strong solutions without any Cho-Choe-Kim type compatibility conditions in $\R^2.$
 Li and Sun \cite{LS2D} obtained the existence of global weak solutions for the 2D non-resistive compressible MHD equations. \cite{LiZh1} extended this global result to include density-dependent viscosity coefficient and non-monotone pressure law. We emphasize that no global well-posedness or stability result is currently available for this 3D MHD system
 in the whole space $\mathbb R^3$ or the periodic box $\mathbb T^3$ even
 when the initial data is small or near a steady-state solution.

 \subsection{Main result}
Attention is focused on the  system \eqref{m2} for $(t,x)\in[0,\infty)\times\mathbb{T}^3$ with the volume of $\mathbb{T}^3$ normalized to unity,
\begin{align*}
\left|\mathbb{T}^3\right|=1.
\end{align*}
For notational convenience, we write
$$
\bar{\rho} = \int_{\mathbb{T}^3}\rho_0\,dx.
$$
Besides the regularity assumption on the initial data $(\rho_0, \u_0, \b_0)$, we also assume
that
\begin{align*}
 \int_{\mathbb{T}^3}\rho_0 \u_0\,dx=0 \quad\mathrm{and}\quad \int_{\mathbb{T}^3} \b_0\,dx=0.
\end{align*}
It is easy to check that, for sufficiently regular solutions, these averages are conserved in time,
\begin{align}\label{conservation}
	\int_{\mathbb{T}^3}\rho \,dx=\bar{\rho}, \quad \int_{\mathbb{T}^3}\rho \u \,dx=0, \quad\mathrm{and}\quad \int_{\mathbb{T}^3} \b \,dx=0.
\end{align}

The main result of the paper is stated as follows.
\begin{theorem}\label{dingli}
Assume $\n$ satisfies the Diophantine condition (\ref{diufantu}).  Let ${N}\ge 4r+7$ with $r>2$.
Consider the system (\ref{m2}) with the initial data $(\rho_0, \u_0,\b_0)$ satisfying
\begin{align*}
	\rho_0-\bar{\rho}\in H^{N}(\T^3),\quad c_0\le\rho_0\le c_0^{-1},\quad\u_0\in H^{N}(\T^3), \quad\b_0\in H^{N}(\T^3)
\end{align*}
for some constant $c_0>0$. We further assume
\begin{align*}
\int_{\T^3}\rho_0\u_0\,dx=\int_{\T^3}\b_0\,dx=0.
\end{align*}
Then there exists a small constant $\varepsilon$ such that, if
\begin{align*}
\norm{\rho_0-\bar{\rho}}{H^{N}}+\norm{\u_0}{H^{N}}+\norm{\b_0}{H^{N}}\le\varepsilon,
\end{align*}
then the system \eqref{m2} admits a unique  global  solution $(\rho-\bar{\rho}, \u,\b)\in C([0,\infty );H^{N})$. Moreover, for any $t\ge 0$ and $r+4\le\beta  \le N $,
there holds
\begin{align}
\norm{(\rho-\bar{\rho})(t)}{H^{\beta}}+\norm{\u(t)}{H^{\beta}}+\norm{\b(t)}{H^{\beta}}\le C \varepsilon (1+t)^{-\frac{3({N}-\beta)}{2({N}-r-4)}}.\label{dec}
\end{align}
\end{theorem}

\vskip .1in
As established in  \cite[Section 2]{zhangzhifei}, almost all vector fields
$\n$ in $\R^3$ actually satisfy the Diophantine condition (\ref{diufantu}). Of course, there are
vectors that do not satisfy this condition such as those with all components being rational numbers.

\vskip .1in
\subsection{Difficulties and scheme of the proof}
The local well-posedness of (\ref{m2}) can be shown via a procedure that is now standard for compressible
equations (see, e.g., \cite{Kawashima}). The focus of the proof is on the global bound of
$(\rho-\bar{\rho}, \u,\b)$ in $H^N(\mathbb T^3)$. We use the bootstrapping argument and start by making the
ansatz that
\begin{align*}
	\sup_{t\in[0,T]}\big(\norm{\rho-\bar{\rho}}{H^{N}}+\norm{\u}{H^{N}}+\norm{\b}{H^{N}}\big)\le \delta,
\end{align*}
for suitably chosen $0<\delta<1$. The main efforts are devoted to proving that, if the initial norm is taken to be sufficiently small, namely
$$
\|\rho_0- \bar\rho\|_{H^N}+\|\u_0\|_{H^N}
 +\|\b_0\|_{H^N}\le \varepsilon
$$
with sufficiently small $\varepsilon>0$, then
\begin{align} \label{task}
\sup_{t\in[0,T]}\big(\norm{\rho-\bar{\rho}}{H^{N}}+\norm{\u}{H^{N}}+\norm{\b}{H^{N}}\big)\le \frac{\delta}{2}.
\end{align}
It is not trivial to prove (\ref{task}). The difficulty is due to the lack of dissipation or damping in the equations of $\rho$ and $\b$. Without any stabilizing mechanism, the norms of $\rho$ and $\b$ would
grow in time and it would be impossible to establish (\ref{task}). Therefore, it is crucial to exploit any
potential smoothing and stabilizing effects due to the coupling and interaction of $\rho$, $\u$ and $\b$.
This paper is able to discover the hidden dissipation in two quantities. The first one is the directional
derivative of $\b$ along $\n$, namely $\n\cdot \nabla \b$ while the second one is the combined
quantity $\rho - \bar \rho + \n\cdot \b$. To be more precise, we assume $\bar \rho=1$ and write $a= \rho -1$ and rewrite (\ref{m2}) as
\begin{eqnarray*}
	\left\{\begin{aligned}
		&\partial_t a+ \div\u  =f_1,\\
		&\partial_t \u-\div(\bar{\mu}(\rho)\nabla\u)-\nabla(\bar{\lambda}(\rho)\div\u)+\nabla a + \nabla ({\n}\cdot\b)-{\n}\cdot\nabla \b
		=f_2,\\
		&\partial_t \b-{\n}\cdot\nabla \u+\n\div\u=f_3,\\
		&\div \b =0,
	\end{aligned}\right.
\end{eqnarray*}
where $\bar{\mu}$, $\bar{\lambda}$, $f_1$, $f_2$ and $f_3$ are explicitly given in Subsection \ref{High}.
$f_1$, $f_2$ and $f_3$ are essentially nonlinear terms. We outline the main steps in the proof of
(\ref{task}) and explain how the coupling and interaction leads to the dissipative and stabilizing effects
in $\n\cdot \nabla \b$ and $a+ \n\cdot \b$.

\vskip .1in
The first step estimates the norm $\|(a, \u,\b)\|_{H^\ell}$ with $0\le \ell\le N$ in terms
of the $L^\infty$-norms of $a,\,\u,\, \b$ and their gradients,
$$
	\frac12\frac{d}{dt}\norm{(a,\u,\b)}{H^{\ell}}^2+\mu\norm{\nabla\u}{H^{\ell}}^2
	+(\lambda+\mu)\norm{\div\u}{H^{\ell}}^2\le C
	Y_{\infty}(t)\norm{(a,\u,\b)}{H^{\ell}}^2,
$$
where $Y_{\infty}(t)$, as given in Lemma \ref{ed2}, defined by
\begin{align}
	Y_{\infty}(t)\stackrel{\mathrm{def}}{=}&
	\norm{(\nabla a,\nabla\u,\nabla\b)}{L^\infty}
	+\norm{(\nabla a,\nabla\u,\nabla\b)}{L^\infty}^2 \notag\\
	&+\|(a,\b)\|_{L^{\infty}}^2+\|a\|_{L^{\infty}}^2\|\b\|_{L^{\infty}}^2+\| \b\|_{L^{\infty}}^2\|\nabla \b\|_{L^{\infty}}^2.\label{gg1}
\end{align}
 In order
to obtain a global bound for $\|(a, \u,\b)\|_{H^\ell}$, we need to control the time integral of $Y_\infty(t)$. Due to the lack of dissipation or damping terms in the equations of $a$ and $\b$,
the time integrals of $\|a\|_{L^\infty}$, $\|\b\|_{L^\infty}$, $\|\nabla a\|_{L^\infty}$ and  $\|\nabla \b\|_{L^\infty}$ can not be bounded directly. Our idea is to exploit the hidden dissipation in
the quantities $\n\cdot \nabla \b$ and $a+ \n\cdot \b$.

\vskip .1in
The second main step is to combine the equation of $\mathbb P \u$ (the divergence part of $\u$) and the
equation of $\b$ to establish the stabilizing effect in $\n\cdot \nabla \b$. Here $\mathbb P=\mathbb I - \nabla \Delta^{-1} \mbox{div}$ is the projection onto divergence-free vector fields. We take advantage of the special structure in the system
$$
\begin{cases}
	\partial_t \mathbb P \u -\mu \Delta  \mathbb P \u = \n\cdot \nabla \b + \mathbb P f_4, \\
	\partial_t \b-{\n}\cdot\nabla \u+\n\div\u=f_3,
\end{cases}
$$
and consider the time evolution of the inner product
$$
\frac{d}{dt}\sum_{0\le s\le {r+3}}\int_{\T^3}\la^{s}\p\u\cdot\la^{s}({\n}\cdot\nabla \b)\,dx.
$$
We obtain, as stated in Lemma \ref{gg}, that
\begin{equation}\label{gg2}
	\norm{{\n}\cdot\nabla \b(t)}{H^{r+3}}^2-\frac{d}{dt}\sum_{0\le s\le {r+3}}\int_{\T^3}\la^{s}\p\u\cdot\la^{s}({\n}\cdot\nabla \b)\,dx\le C\norm{ \u(t)}{H^{r+5}}^2,
\end{equation}
which yields the time integrability of $\norm{{\n}\cdot\nabla \b(t)}{H^{r+3}}^2$. Combining with the
Poincar\'e type inequality induced by the Diophantine condition on $\n$,
$$
\|\b\|_{H^3} \le C\, \|\n\cdot \nabla \b\|_{H^{r+3}},
$$
we are then able to control the time integrability of $\|\b\|_{L^\infty}$ and $\|\nabla \b\|_{L^\infty}$.

\vskip .1in
The third main step is to establish the stabilizing property of the combined quantity $a + \n \cdot \b$.
The idea is to make use of the interaction between the equation of $a + \n \cdot \b$ and that of
$\mathbb Q \u$,
\begin{equation}\label{ww}
\begin{cases}
	\partial_t (a+{\n}\cdot\b)={\n}\cdot\nabla \u\cdot{\n}-(|\n|^2+1)\div\u+f_1+f_3\cdot{\n},\\
	\partial_t \q\u-\nu\Delta \q\u+\nabla (a +{\n}\cdot\b)=\q f_4
\end{cases}
\end{equation}
where $\mathbb Q= \nabla\Delta^{-1} \mbox{div}$ projects vectors onto their gradient parts. Noting that
$\mbox{div}  \q\u = \div\u$, we observe that (\ref{ww}) involves a wave structure for $a + \n \cdot \b$
and $\q\u$. To make the wave structure more explicit, we can rewrite (\ref{ww}) as, after ignoring
some non-essential terms,
$$
	\begin{cases}
		\partial_t (a+{\n}\cdot\b)=-(|\n|^2+1)\div\u,\\
		\partial_t \q\u-\nu\Delta \q\u+\nabla (a +{\n}\cdot\b)=0,
	\end{cases}
$$
which can be converted into
$$
\begin{cases}
	\partial_{tt} (a+{\n}\cdot\b)- \nu \Delta \partial_t (a+{\n}\cdot\b) -(|\n|^2+1)\Delta (a+{\n}\cdot\b) =0,  \\
	\partial_{tt} \mbox{div}  \q\u-\nu \Delta \partial_t \mbox{div}  \q\u -(|\n|^2+1)\Delta \mbox{div}  \q\u =0.
\end{cases}
$$
Alternatively we can make the stabilizing effect more explicit by considering the combined quantities
$$
	{d}\stackrel{\mathrm{def}}{=} a + {{\n}\cdot\b}\quad\mbox{and}\quad
	\G\stackrel{\mathrm{def}}{=}\q\u-\frac1\nu\Delta^{-1}\nabla {d},
$$
which satisfy
\begin{eqnarray*}
	\left\{\begin{aligned}
		&\partial_t d +\frac1\nu(|\n|^2+1){d}+ (|\n|^2+1)\div \G  ={\n}\cdot\nabla \u\cdot{\n}+f_1+f_3\cdot{\n},\\
		&\partial_t \G-\nu\Delta \G =\frac1\nu(|\n|^2+1)\q\u
		-\frac1\nu\Delta^{-1}\nabla ({\n}\cdot\nabla \u\cdot{\n})+\q f_4-\frac1\nu\Delta^{-1}\nabla (f_1+f_3\cdot{\n}).
	\end{aligned}\right.
\end{eqnarray*}
These equations clearly reveal the damping in $d=a + {{\n}\cdot\b}$ and the dissipation in $\G$.
We then estimate the Sobolev norms $\|d\|_{H^{r+4}}$ and $\|\G\|_{H^{r+4}}$. $a$ in some of the nonlinear terms in $f_1$ and $f_4$ is replaced by $d -{{\n}\cdot\b}$. Combining with (\ref{gg1}) with $\ell=r+4$ and (\ref{gg2}) in the
previous steps, we then obtain a self-contained inequality.

\vskip .1in
The last step is to establish (\ref{task}) and close the bootstrapping argument. The energy inequality
obtained in the previous step, together with interpolation inequalities, allow us to show that
\begin{equation}\label{gg5}
\|(a, \u, \b)(t)\|_{H^3} \le C\, (1+ t)^{-\frac32}
\end{equation}
when $\delta>0$ is taken to be sufficiently small. In particular, the time integral of $Y_\infty(t)$ in (\ref{gg1}) is time integrable,
$$
\int_0^\infty Y_\infty(t) \,dt \le C <\infty.
$$
Gr\"onwall's inequality then yields
$$
\|(a, \u, \b)(t)\|_{H^N} \le C\, \|(a_0, \u_0,\b_0)\|_{H^N}.
$$
Taking the initial norm to be sufficiently small, we achieve (\ref{task}). The decay rate in (\ref{dec})
is a consequence of (\ref{gg5}) and an interpolation inequality.

\vskip .1in
We point out some differences between the handling of this compressible MHD system and that of the
corresponding incompressible counterpart. The compressible MHD system contains the density, which
involves no damping or dissipation. A crucial component in the proof of Theorem \ref{dingli} is
how to obtain the decay estimate for $\|a(t)\|_{H^3}$. It does not appear to be possible to work
directly with the equation of $a$ to obtain such decay estimate. As described above, our idea is to
exploit the wave structure in the evolution of $a+{\n}\cdot\b$ and $\mbox{div}  \q\u$. In contrast,
the incompressible MHD equations do not contain the density equation. The second main difference is that the velocity $\u$ of the compressible MHD system contains the divergence-free part and the gradient part,
$$
\u =\mathbb P \u + \mathbb Q \u.
$$
It is necessary to consider the equations of $\mathbb P \u$ and $\mathbb Q \u$ separately when we seek the hidden stabilizing effect in the quantities $\n\cdot \nabla \b$ and $a+ \n\cdot\b$. The velocity $\u$
in the incompressible MHD system is already divergence-free and does not have the gradient part. Thus the
compressible MHD system is much more challenging to deal with.

\vskip .1in
\subsection{Organization of the paper}
The rest of this paper is structured as follows. Section \ref{ineq} recalls several
functional inequalities to be used in the proof of Theorem \ref{dingli}. Section \ref{proo}
proves Theorem \ref{dingli}. The long proof is accomplished in six subsections. Subsection \ref{loc}
explains how to prove the local well-posedness and initiates the bootstrapping argument.
Subsection \ref{bas} provides the basic $L^2$-energy estimate. Subsection \ref{High} establishes
the higher-order energy estimate (\ref{gg1}). Subsection \ref{dis} explores the dissipation in $\b$ and proves (\ref{gg2}). Subsection \ref{dis2} discovers the dissipation of the quantity $a+ \n\cdot \b$
and obtains a self-contained inequality for $(a, \u, \b)$ in $H^{r+4}$.  Subsection \ref{compl} proves
the decay estimate (\ref{gg5}) and then (\ref{task}) to close the bootstrapping argument.
The decay rate in (\ref{dec})
is shown via (\ref{gg5}) and an interpolation inequality.

\vskip .3in
\section{Preliminaries}
\label{ineq}

This section provides  several functional inequalities to be used in the proof of our main result.
We first recall a weighted Poincar\'e inequality first established by Desvillettes and Villani in \cite{DV05}.
\begin{lemma}\label{lem-Poi}
	Let $\Omega$  be a bounded connected Lipschitz domain and $\bar{\varrho}$ be a positive constant.  There exists a positive constant $C$, depending on $\Omega$ and $\bar{\varrho}$, such that  for any nonnegative function $\varrho$ satisfying
	\begin{align*}
		\int_{\Omega}\varrho dx=1, \quad \varrho\le\bar{\varrho},
	\end{align*}
	and any $\u\in H^1(\Omega)$, there holds
	\begin{align}\label{wPoi}
		\int_{\Omega}\varrho \left(\u-\int_{\mathbb{T}^3}\varrho \u dx\right)^2\,dx\le C\|\nabla \u\|_{L^2}^2.
	\end{align}
\end{lemma}

\vskip .1in
In order to remove the weight function $\varrho$ in  \eqref{wPoi} without resorting to the lower bound of  $\varrho$, we need  another variant of Poincar\'e inequality (see Lemma 3.2 in \cite{F04}).
\begin{lemma}\label{lem2.2}
	Let $\Omega$ be a bounded connected Lipschitz domain in $\mathbb{R}^3$ and
	$p>1$ be a constant. Given positive constants $M_0$ and $E_0$, there is a constant $C=C(E_0,M_0)$ such that for any
	non-negative function $\varrho$ satisfying
	$$
	M_0\leq\int_{\Omega}\varrho dx\quad\mbox{and}\quad  \int_{\Omega}\varrho^{p}dx\leq
	E_0,
	$$
	and for any $\u\in H^1(\Omega)$, there holds
	$$
	\|\u\|_{L^2}^2\leq C\left[\|\nabla
	\u\|_{L^2}^2+\left(\int_{\Omega}\varrho|\u|dx\right)^2\right].
	$$
\end{lemma}

\vskip .1in
The next lemma states a special Poincar\'e inequality involving a vector satisfying the Diophantine condition.

\begin{lemma}\label{diu}
	Assume  ${\n}\in\R^3$ satisfies the Diophantine condition \eqref{diufantu}. Let $s\in \mathbb R$. Then, for any function $f$ satisfies  $\nabla f\in H^{s+r}(\T^3)$ and $\int_{\T^3}f\,dx=0$,
	\begin{equation*}
		\|f\|_{H^{s}(\T^3)}\le C\|{\n}\cdot\nabla f\|_{H^{s+r}(\T^3)}.
	\end{equation*}
\end{lemma}

The proof of this lemma follows from the standard Poincar\'e inequality and the Diophantine condition.
Finally we recall several calculus inequalities.

\begin{lemma}\label{daishu}{\rm(\cite{kato})}
Let $s\ge 0$. Then there exists a constant $C$ such that, for any $f,g\in {H^{s}}(\T^3)\cap {L^\infty}(\T^3)$, we have
\begin{equation*}
\|fg\|_{H^{s}}\le C(\|f\|_{L^\infty}\|g\|_{H^{s}}+\|g\|_{L^\infty}\|f\|_{H^{s}}).
\end{equation*}
\end{lemma}

\begin{lemma}\label{jiaohuanzi}{\rm(\cite{kato})}
Let $s> 0$. Then there exists a constant $C$ such that, for any $f\in {H^{s}}(\T^3)\cap W^{1,\infty}(\T^3)$, $g\in {H^{s-1}}(\T^3)\cap {L^\infty}(\T^3)$, there holds
\begin{align*}%\label{ping6}
\norm{[\la^s,f\cdot\nabla ]g}{L^2}\le C(\norm{\nabla f}{L^\infty}\norm{\la^sg}{L^2}+\norm{\la^s f}{L^2}\norm{\nabla g}{L^\infty}).
\end{align*}
\end{lemma}

\begin{lemma}\label{fuhe}{\rm(\cite{Triebel})}
Let $s>0$ and $f\in H^s(\T^3)\cap L^\infty(\T^3)$. Assume that $F$ is a smooth function  on $\R$ with $F(0)=0$. Then we
have
$$
\|F(f)\|_{H^s}\le C(1+\|f\|_{L^\infty})^{[s]+1}\|f\|_{H^s},
$$
where the constant $C$ depends on $\sup_{k\le{[s]+2},t\le\|f\|_{L^\infty}} \|F^{(k)}(t)\|_{L^\infty}.$
\end{lemma}

\vskip .3in
\section{Proof of the main theorem}
\label{proo}

This section is devoted to proving Theorem \ref{dingli}. The proof is long
and is thus divided into several
subsections for the sake of clarity.

\vskip .1in
\subsection{Local well-posedness}\label{loc}
Given the initial data $(\rho_0-\bar{\rho},\u_0, \b_0)\in H^{N}(\T^3)$, the local well-posedness of \eqref{m2} could be proven
by using the standard energy method (see, e.g.,  \cite{Kawashima}). Thus, we may assume that there exists $T > 0$  such that the system \eqref{m2} has a unique solution
$(\rho-\bar{\rho},\u,\b)\in C([0,T];H^{N})$. Moreover,
\begin{align}\label{youjiexing}
\frac12c_0\le\rho(t,x)\le 2c_0^{-1},\quad\hbox{for any $t\in[0,T]$}.
\end{align}
We use the bootstrapping argument to show that this local solution can be extended into a global one.
The goal is to derive a global a priori upper bound. To initiate the bootstrapping argument, we make the
ansatz that
\begin{align}\label{ping23}
\sup_{t\in[0,T]}(\norm{\rho-\bar{\rho}}{H^{N}}+\norm{\u}{H^{N}}+\norm{\b}{H^{N}})\le \delta,
\end{align}
where $0<\delta<1$ obeys requirements to be specified later. In the following subsections we prove that, if the initial norm is taken to be sufficiently small, namely
$$
\|a_0\|_{H^N}+\|\u_0\|_{H^N}+\|\b_0\|_{H^N} \le \varepsilon, \quad a_0 = \rho_0- \bar\rho
$$
with sufficiently small $\varepsilon>0$, then
$$
\sup_{t\in[0,T]}(\norm{\rho-\bar{\rho}}{H^{N}}+\norm{\u}{H^{N}}+\norm{\b}{H^{N}})\le \frac{\delta}{2}.
$$
The bootstrapping argument then leads to the desired global bound.

\subsection{Basic energy estimates} \label{bas}
Denote by $g(\rho)$ the potential energy density, namely
\begin{align*}
g(\rho)=\rho\int_{\bar{\rho}}^\rho\frac{P(\tau)-P(\bar{\rho})}{\tau^2}\,d\tau.
\end{align*}
For any fixed positive constant $c_0$, if $c_0\le\rho\le c_0^{-1}$, then
\begin{align*}
g(\rho)\sim (\rho-\bar{\rho})^2.
\end{align*}
The standard basic energy estimate gives
\begin{align}\label{ping3}
\frac12\frac{d}{dt}\int_{\T^3}\left(2g(\rho)+\rho|\u|^2+|\b|^2\right)\,dx+\mu\norm{\nabla\u}{L^2}^2+(\lambda+\mu)\norm{\div\u}{L^2}^2=0,
\end{align}
where we have used the following cancellations
\begin{align*}
&\int_{\T^3}({\n}\cdot\nabla \b+\b\cdot\nabla\b)\cdot\u\,dx+\int_{\T^3}({\n}\cdot\nabla \u+\b\cdot\nabla\u)\cdot\b\,dx=0,\nn\\
&\int_{\T^3}(\b\nabla\b+{\n}\nabla \b)\cdot\u\,dx+\int_{\T^3}(\u\cdot\nabla\b+\n\div\u+\b\div\u)\cdot\b\,dx=0.
\end{align*}
Without loss of generality, we assume from now on that $\bar{\rho}=1$.  If we set
 $$a\stackrel{\mathrm{def}}{=}\rho-1,$$
then \eqref{ping3} implies the result in the following lemma.
\begin{lemma}\label{ping5}
Let $\rho, \u$ and $\b$ be smooth functions to
	(\ref{m2}), then there holds
\begin{align*}
\frac12\frac{d}{dt}\norm{(a,\u,\b)}{L^2}^2+\mu\norm{\nabla\u}{L^2}^2+(\lambda+\mu)\norm{\div\u}{L^2}^2=0.
\end{align*}
\end{lemma}

In view of Lemmas \ref{lem-Poi}, \ref{lem2.2}  and \eqref{conservation}, we have the following Poincar\'e type inequalities, which will be frequently used in this paper.

\begin{lemma}[Poincar\'e type inequalities]%\label{lem3.3}
Let $\rho, \u$ and $\b$ be smooth functions to
	(\ref{m2}) on $[0,\infty)\times\mathbb{T}^3$ satisfying \eqref{conservation}, then for any $t\ge0$, there hold
\begin{align}\label{Poi1H}
\|\b(t)\|_{L^2}^2\le C\|\nabla \b(t)\|_{L^2}^2,
\end{align}

\begin{align}\label{Poi1}
\|(\sqrt{\rho }\u)(t)\|_{L^2}^2\le C\|\nabla \u(t)\|_{L^2}^2,
\end{align}
and
\begin{align}\label{Poi1'}
\|\u(t)\|_{L^2}^2\le C\|\nabla \u(t)\|_{L^2}^2.
\end{align}
\end{lemma}
\begin{proof}
Clearly, Due to \eqref{conservation}, \eqref{Poi1H} is a standard Poincar\'e inequality.
\eqref{Poi1} follows from Lemma \ref{lem-Poi} and $\int_{\mathbb{T}^3}\rho \u dx=0$ in \eqref{conservation}.  \eqref{Poi1'} is a consequence of \eqref{Poi1} and Lemma \ref{lem2.2}.
\end{proof}

Throughout we make the assumption that
\begin{equation}\label{eq:smallad}
\sup_{t\in\R_+,\, x\in\T^3} |a(t,x)|\leq \frac12.
\end{equation}
Because of $H^2(\T^3)\hookrightarrow L^\infty(\T^3),$ \eqref{eq:smallad} is ensured by the fact that the  solution constructed here has small norm in $H^2(\T^3)$. It then follows from Lemma \ref{fuhe} that
the following composition estimate holds,
\begin{equation}\label{eq:smalla}
\|F(a)\|_{H^s}\le C\|a\|_{H^s}, \quad\hbox{for $F(0)=0$ and any $s>0$.}
\end{equation}

\subsection{Higher order energy estimates} \label{High}

To obtain higher order energy estimates, we set
$$ \bar{\mu}(\rho)\stackrel{\mathrm{def}}{=}\frac{\mu}{\rho},\quad\bar{\lambda}(\rho)\stackrel{\mathrm{def}}{=}\frac{\lambda+\mu}{\rho},\quad I(a)\stackrel{\mathrm{def}}{=}\frac{a}{1+a},\quad\hbox{and}\quad k(a)\stackrel{\mathrm{def}}{=}\frac{P'(1+a)}{1+a}-1.$$
Then \eqref{m2} can be reformulated as
\begin{eqnarray}\label{m3}
\left\{\begin{aligned}
&\partial_t a+ \div\u  =f_1,\\
&\partial_t \u-\div(\bar{\mu}(\rho)\nabla\u)-\nabla(\bar{\lambda}(\rho)\div\u)+\nabla a
={\n}\cdot\nabla \b-\nabla ({\n}\cdot\b)+f_2,\\
&\partial_t \b={\n}\cdot\nabla \u-\n\div\u+f_3,\\
&\div \b =0,\\
&(a,\u,\b)|_{t=0}=(a_0,\u_0,\b_0),
\end{aligned}\right.
\end{eqnarray}
where
\begin{align*}
f_1\stackrel{\mathrm{def}}{=}&-\u\cdot\nabla a-a\div \u,\nn\\
f_2\stackrel{\mathrm{def}}{=}&-\u\cdot\nabla \u+\b\cdot\nabla\b+\b\nabla\b+k(a)\nabla a+\mu(\nabla I(a))\nabla\u
\nn\\
&+(\lambda+\mu)(\nabla I(a))\div\u-I(a)({\n}\cdot\nabla \b+\b\cdot\nabla\b-{\n}\nabla \b-\b\nabla\b),\nn\\
f_3\stackrel{\mathrm{def}}{=}&-\u\cdot\nabla\b+\b\cdot\nabla\u-\b\div\u.
\end{align*}
The aim of this subsection is to prove the following key lemma.

\begin{lemma}\label{ed2}
Let $(a,\u,\b) \in C([0, T];H^N)$ be a solution to the  system \eqref{m3}. For any $0\le \ell\le N$, there holds
\begin{align}\label{ed3}
&\frac12\frac{d}{dt}\norm{(a,\u,\b)}{H^{\ell}}^2+\mu\norm{\nabla\u}{H^{\ell}}^2
+(\lambda+\mu)\norm{\div\u}{H^{\ell}}^2\le C
Y_{\infty}(t)\norm{(a,\u,\b)}{H^{\ell}}^2
\end{align}
with
\begin{align*}
Y_{\infty}(t)\stackrel{\mathrm{def}}{=}&
\norm{(\nabla a,\nabla\u,\nabla\b)}{L^\infty}
+\norm{(\nabla a,\nabla\u,\nabla\b)}{L^\infty}^2\notag\\
&+\|(a,\b)\|_{L^{\infty}}^2 +\|a\|_{L^{\infty}}^2\|\b\|_{L^{\infty}}^2+\| \b\|_{L^{\infty}}^2\|\nabla \b\|_{L^{\infty}}^2.
\end{align*}
\end{lemma}
\begin{proof}
(\ref{ed3}) with $\ell=0$ is the basic energy inequality in Lemma \ref{ping5}. We consider the case when $\ell\ge 1$.
Writing $\Lambda=\sqrt{-\Delta}$ and applying  $\la^s$ with $1\le s\le \ell$ to \eqref{m3} and then taking $L^2$ inner product with $(\la^sa, \la^s\u,\la^s\b)$ yields
\begin{align}\label{ed356}
&\frac12\frac{d}{dt}\norm{(\la^sa,\la^s\u,\la^s\b)}{L^2}^2
-\int_{\T^3}\la^{s}\div(\bar{\mu}(\rho)\nabla\u)\cdot\la^{s} \u\,dx-\int_{\T^3}\la^{s}\nabla(\bar{\lambda}(\rho)\div\u)\cdot\la^{s} \u\,dx\nn\\
&\quad= \int_{\T^3}\la^{s} f_1\cdot\la^{s} a\,dx+\int_{\T^3}\la^{s} f_2\cdot\la^{s} \u\,dx+\int_{\T^3}\la^{s} f_3\cdot\la^{s} \b\,dx
\end{align}
where we have used the following cancellations
\begin{align*}
&\int_{\T^3}\la^s \div\u\cdot\la^s a\,dx+\int_{\T^3}\la^s \nabla a\cdot\la^s \u\,dx=0;\nn\\
&\int_{\T^3}\la^s({\n}\cdot\nabla \b)\cdot\la^s \u\,dx+\int_{\T^3}\la^s ({\n}\cdot\nabla \u)\cdot\la^s \b\,dx=0;\nn\\
&\int_{\T^3}\la^s\nabla ({\n}\cdot\b)\cdot\la^s \u\,dx+\int_{\T^3}\la^s (\n\div\u)\cdot\la^s \b\,dx=0.
\end{align*}
The second term on the left-hand side of \eqref{ed356} can be written as
\begin{align}\label{bianx1}
&-\int_{\T^3}\la^{s}\div(\bar{\mu}(\rho)\nabla\u)\cdot\la^{s} \u\,dx\nn\\
&\quad
=\int_{\T^3}\la^{s}(\bar{\mu}(\rho)\nabla\u)\cdot\nabla\la^{s} \u\,dx\nn\\
&\quad= \int_{\T^3}\bar{\mu}(\rho)\nabla\la^{s}\u\cdot\nabla\la^{s} \u\,dx+\int_{\T^3}[\la^{s},\bar{\mu}(\rho)]\nabla\u\cdot\nabla\la^{s} \u\,dx.
\end{align}
Due to \eqref{youjiexing}, we have for any $t\in[0,T]$ that
\begin{align}\label{bianx2}
 \int_{\T^3}\bar{\mu}(\rho)\nabla\la^{s}\u\cdot\nabla\la^{s} \u\,dx\ge c_0^{-1}\mu\norm{\la^{s+1}\u}{L^{2}}^2.
\end{align}
For the last term in \eqref{bianx1}, we first rewrite this term into
\begin{align*}%\label{bianx3}
\int_{\T^3}[\la^{s},\bar{\mu}(\rho)]\nabla\u\cdot\nabla\la^{s} \u\,dx
=&\int_{\T^3}[\la^{s},\bar{\mu}(\rho)-\mu+\mu]\nabla\u\cdot\nabla\la^{s} \u\,dx\nn\\
=&-\int_{\T^3}[\la^{s},\mu I(a)]\nabla\u\cdot\nabla\la^{s} \u\,dx.
\end{align*}
Then, with the aid of Lemmas \ref{jiaohuanzi}, \ref{fuhe} and \eqref{eq:smalla}, we have
\begin{align}\label{bianx5}
&\Big|\int_{\T^3}[\la^{s},\mu I(a)]\nabla\u\cdot\nabla\la^{s} \u\,dx\Big| \nn\\
&\quad \le C\norm{\nabla\la^{s}\u}{L^{2}}(\norm{\nabla I(a) }{L^\infty}\norm{\la^{s} \u}{L^2}+\norm{\nabla \u}{L^\infty}\norm{\la^{s}I(a)}{L^2})\nn\\
&\quad \le \frac{c_0^{-1}}{2}\mu\norm{\la^{s+1}\u}{L^{2}}^2+C\big(\norm{\nabla a }{L^\infty}^2\norm{\la^{s} \u}{L^2}^2+\norm{\nabla \u}{L^\infty}^2\norm{\la^{s}a}{L^2}^2\big).
\end{align}
Inserting \eqref{bianx2} and \eqref{bianx5} into \eqref{bianx1} leads to
\begin{align*}%\label{bianx6}
-\int_{\T^3}\la^{s}\div(\bar{\mu}(\rho)\nabla\u)\cdot\la^{s} \u\,dx
\ge&\frac{c_0^{-1}}{2}\mu\norm{\la^{s+1}\u}{L^{2}}^2\nn\\
&-C\big(\norm{\nabla a }{L^\infty}^2\norm{\la^{s} \u}{L^2}^2+\norm{\nabla \u}{L^\infty}^2\norm{\la^{s}a}{L^2}^2\big).
\end{align*}
The third term on the left-hand side of  \eqref{ed356} can be dealt with similarly.  Hence,
\begin{align}\label{ed4}
&\frac12\frac{d}{dt}\norm{(\la^sa,\la^s\u,\la^s\b)}{L^2}^2
+{c_0^{-1}}\mu\norm{\la^{s+1}\u}{L^2}^2 +{c_0^{-1}}(\lambda+\mu)\norm{\la^{s}\mbox{div}\u}{L^2}^2 \nn\\
&\quad\le
C\big(\norm{\nabla a }{L^\infty}^2\norm{\la^{s} \u}{L^2}^2+\norm{\nabla \u}{L^\infty}^2\norm{\la^{s}a}{L^2}^2\big)\nn\\
&\qquad+
\int_{\T^3}\la^{s} f_1\cdot\la^{s} a\,dx+\int_{\T^3}\la^{s} f_2\cdot\la^{s} \u\,dx+\int_{\T^3}\la^{s} f_3\cdot\la^{s} \b\,dx.
\end{align}
We now estimate successively terms on the right hand side of \eqref{ed4}.
To bound the first term in $f_1$, we rewrite it into
\begin{align}\label{ed5-1}
\int_{\T^3}\la^{s} (\u\cdot\nabla a)\cdot\la^{s} a\,dx
=&\int_{\T^3}(\la^{s} (\u\cdot\nabla a)-\u\cdot\nabla\la^{s}a)\cdot\la^{s} a\,dx
+\int_{\T^3}\u\cdot\nabla\la^{s}a\cdot\la^{s} a\,dx\nn\\
\stackrel{\mathrm{def}}{=}&A_1+A_2.
\end{align}
By Lemma \ref{jiaohuanzi},
\begin{align}\label{we1}
A_1\le& C\norm{[\la^s,\u\cdot\nabla ]a}{L^2}\norm{\la^{s}a}{L^2}\nn\\
\le&C(\norm{\nabla \u}{L^\infty}\norm{\la^{s}a}{L^2}+\norm{\la^{s} \u}{L^2}\norm{\nabla a}{L^\infty})\norm{\la^{s}a}{L^2}\nn\\
\le&C(\norm{\nabla \u}{L^\infty}+\norm{\nabla a}{L^\infty})(\norm{\la^{s}a}{L^2}^2+\norm{\la^{s}\u}{L^2}^2).
\end{align}
By integration by parts,
\begin{align}\label{we2}
A_2\le& C\norm{\nabla \u}{L^\infty}\norm{\la^{s}a}{L^2}^2.
\end{align}
For the second term in $f_1$, it follows from Lemma \ref{daishu} that
\begin{align}\label{ed6}
\int_{\T^3}\la^{s} (a\div\u)\cdot\la^{s} a\,dx\leq&C\big(\|\div\u\|_{L^{\infty}}\|a\|_{H^{s}}
        +\|\div\u\|_{H^{s}}\|a\|_{L^{\infty}}\big)\norm{\la^{s}a}{L^2}\nonumber\\
         \leq&\frac{c_0^{-1}\mu}{16}\|\la^{s+1}\u\|_{L^2}^2+C(\|\nabla\u\|_{L^{\infty}}+\|a\|_{L^{\infty}}^2)
         \norm{\la^{s}a}{L^2}^2,
\end{align}
where we have used the inequalities
\begin{align*}
\|a\|_{H^{s}}\le C\norm{\la^{s}a}{L^2}\quad\mbox{and}\quad  \|\div\u\|_{H^{s}}\le C\|\la^{s+1}\u\|_{L^{2}}.
\end{align*}
Collecting \eqref{we1}, \eqref{we2} and \eqref{ed6}, we can get
\begin{align}\label{we6}
\int_{\T^3}\la^{s} f_1\cdot\la^{s} a\,dx
\le&\frac{c_0^{-1}\mu}{16}\|\la^{s+1}\u\|_{L^2}^2\nn\\
&+C(\norm{\nabla \u}{L^\infty}+\norm{\nabla a}{L^\infty}+\|a\|_{L^{\infty}}^2)(\norm{\la^{s}a}{L^2}^2+\norm{\la^{s}\u}{L^2}^2).
\end{align}
For the first term in $f_3$, a similar process as in \eqref{we1} and \eqref{we2} yields
\begin{align*}
\int_{\T^3}\la^{s} (\u\cdot\nabla \b)\cdot\la^{s} \b\,dx
\le&C(\norm{\nabla \u}{L^\infty}+\norm{\nabla \b}{L^\infty})(\norm{\la^{s}\u}{L^2}^2+\norm{\la^{s}\b}{L^2}^2).
\end{align*}
For the last two terms in $f_3$, a derivation similar to \eqref{ed6} gives
\begin{align*}
\int_{\T^3}\la^{s} (\b\cdot\nabla\u-\b\div\u)\cdot\la^{s} \b\,dx
         \leq&\frac{c_0^{-1}\mu}{16}\|\la^{s+1}\u\|_{L^2}^2+C(\|\nabla\u\|_{L^{\infty}}+\|\b\|_{L^{\infty}}^2)
         \norm{\la^{s}\b}{L^2}^2.
\end{align*}
Therefore,
\begin{align}\label{we9}
\int_{\T^3}\la^{s} f_3\cdot\la^{s} \b\,dx
\le&\frac{c_0^{-1}\mu}{16}\|\la^{s+1}\u\|_{L^2}^2\nn\\
&+C(\norm{\nabla \u}{L^\infty}+\norm{\nabla \b}{L^\infty}+\|\b\|_{L^{\infty}}^2)(\norm{\la^{s}\u}{L^2}^2+\norm{\la^{s}\b}{L^2}^2).
\end{align}
In the following, we bound the terms in $f_2$. To do so, we  write
\begin{align}\label{we10}
\int_{\T^3}\la^{s} f_2\cdot\la^{s} \u\,dx=\sum_{i=3}^9A_i
\end{align}
with
\begin{eqnarray*}
&&A_3\stackrel{\mathrm{def}}{=}\int_{\T^3}\la^{s} (\u\cdot\nabla \u)\cdot\la^{s} \u\,dx,\ \qquad
A_4\stackrel{\mathrm{def}}{=}\int_{\T^3}\la^{s} (\b\cdot\nabla\b)\cdot\la^{s} \u\,dx,\nn\\
&&A_5\stackrel{\mathrm{def}}{=}\int_{\T^3}\la^{s} (k(a)\nabla a)\cdot\la^{s} \u\,dx,\qquad
A_6\stackrel{\mathrm{def}}{=}\int_{\T^3}\la^{s} (\mu(\nabla I(a))\nabla\u)\cdot\la^{s} \u\,dx,\nn\\
&&A_7\stackrel{\mathrm{def}}{=}\int_{\T^3}\la^{s}( (\lambda+\mu)(\nabla I(a))\div\u)\cdot\la^{s} \u\,dx,\notag\\
&&A_8\stackrel{\mathrm{def}}{=}\int_{\T^3}\la^{s} (I(a)({\n}\cdot\nabla \b-{\n}\nabla \b))\cdot\la^{s} \u\,dx,\nn\\
&&A_9\stackrel{\mathrm{def}}{=}\int_{\T^3}\la^{s} (I(a)(\b\cdot\nabla\b-\b\nabla\b))\cdot\la^{s} \u\,dx.
\end{eqnarray*}
The term $A_3$ can be bounded as in \eqref{ed5-1} to get
\begin{align*}
A_3\le C\norm{\nabla \u}{L^\infty}\norm{\la^{s}\u}{L^2}^2.
\end{align*}
We next deal with the term $A_4$. In view of $\div\b=0,$ one can write
\begin{align*}
A_4=&\int_{\T^3}\la^{s} \div(\b\otimes\b)\cdot\la^{s} \u\,dx\nn\\
\leq&C\|\b\|_{L^{\infty}}\|\b\|_{H^{s}}
        \norm{\la^{s+1}\u}{L^2}\nonumber\\
         \leq&\frac{c_0^{-1}\mu}{16}\|\la^{s+1}\u\|_{L^2}^2+C\|\b\|_{L^{\infty}}^2
         \|\b\|_{H^{s}}^2.
\end{align*}
By Lemmas \ref{daishu} and (\ref{eq:smalla}), we have
\begin{align*}
A_5%=\int_{\T^3}\la^{s} (k(a)\nabla a)\cdot\la^{s} \u\,dx
\leq&C\big(\|\nabla a\|_{L^{\infty}}\|k(a)\|_{H^{s-1}}
        +\|\nabla a\|_{H^{s-1}}\|k(a)\|_{L^{\infty}}\big)\norm{\la^{s+1}\u}{L^2}\nonumber\\
         \leq&\frac{c_0^{-1}\mu}{16}\norm{\la^{s+1}\u}{L^2}^2+C(\|\nabla a\|_{L^{\infty}}^2+\|a\|_{L^{\infty}}^2)
         \|a\|_{H^{s}}^2,
\end{align*}
where we have used the fact that $\|k(a)\|_{L^\infty} \le \|a\|_{L^\infty}$.
Similarly,
\begin{align*}
A_6+A_7%=\int_{\T^3}\la^{s} ( I(a)({\n}\cdot\nabla \b-{\n}\nabla \b))\cdot\la^{s} \u\,dx
\leq&C\big(\|\nabla I(a)\|_{L^{\infty}}\norm{\la^{s}\u}{L^2}
        +\|\nabla I(a)\|_{H^{s-1}}\|\nabla \u\|_{L^{\infty}}\big)\norm{\la^{s+1}\u}{L^2}\nonumber\\
         \leq&\frac{c_0^{-1}\mu}{16}\norm{\la^{s+1}\u}{L^2}^2+C\big(\|\nabla a\|_{L^{\infty}}^2\norm{\la^{s}\u}{L^2}^2
        +\|\nabla \u\|_{L^{\infty}}^2\|a\|_{H^{s}}^2\big)
\end{align*}
and
\begin{align}\label{ed13}
A_8%=\int_{\T^3}\la^{s} (\nabla I(a)\nabla \u)\cdot\la^{s} \u\,dx
\leq&C\big(\|I(a)\|_{L^{\infty}}\|\nabla \b\|_{H^{s-1}}
        +\|I(a)\|_{H^{s-1}}\|\nabla \b\|_{L^{\infty}}\big)\norm{\la^{s+1}\u}{L^2}\nonumber\\
         \leq&\frac{c_0^{-1}\mu}{16}\norm{\la^{s+1}\u}{L^2}^2+C\big(\|a\|_{L^{\infty}}^2\|\b\|_{H^{s}}^2
        +\|\nabla \b\|_{L^{\infty}}^2\|a\|_{H^{s}}^2\big).
\end{align}
For the last term $A_9$, we use Lemmas \ref{daishu} and \ref{fuhe} again to get
\begin{align}\label{ed14}
A_9%=\int_{\T^3}\la^{s} ( I(a)(\b\cdot\nabla\b-\b\nabla\b))\cdot\la^{s} \u\,dx
\leq&C\big(\|I(a)\|_{L^{\infty}}\|\b\nabla \b\|_{H^{s-1}}
        +\|I(a)\|_{H^{s-1}}\|\b\nabla \b\|_{L^{\infty}}\big)\norm{\la^{s+1}\u}{L^2}.
\end{align}
Due to
\begin{align*}
\|\b\nabla \b\|_{H^{s-1}}\le C\|\b\|_{L^{\infty}}\| \b\|_{H^{s}},
\end{align*}
which, together with \eqref{ed14}, leads to
\begin{align*}
A_9 \leq&\frac{c_0^{-1}\mu}{16}\norm{\la^{s+1}\u}{L^2}^2
+C\big(\|a\|_{L^{\infty}}^2\|\b\|_{L^{\infty}}^2\| \b\|_{H^{s}}^2
        +\| \b\|_{L^{\infty}}^2\|\nabla \b\|_{L^{\infty}}^2\| a\|_{H^{s}}^2\big).
\end{align*}
Inserting the bounds for $A_3$ through $A_9$ in \eqref{we10}, we get
\begin{align}\label{ed15}
\int_{\T^3}\la^{s} f_2\cdot\la^{s} \u\,dx
\leq&\frac{5c_0^{-1}\mu}{16}\norm{\la^{s+1}\u}{L^2}^2+C\big(\norm{\nabla \u}{L^\infty}+\|(\nabla a,\nabla \u,\nabla \b)\|_{L^{\infty}}^2+\|(a,\b)\|_{L^{\infty}}^2\nn\\
&\quad+\|a\|_{L^{\infty}}^2\|\b\|_{L^{\infty}}^2
        +\| \b\|_{L^{\infty}}^2\|\nabla \b\|_{L^{\infty}}^2\big)\| (a,\u,\b)\|_{H^{s}}^2.
\end{align}
Plugging \eqref{we6}, \eqref{we9} and \eqref{ed15} into \eqref{ed4} and combining with Lemma \ref{ping5}, we can arrive at the desired estimate \eqref{ed3} by summing up for any $1\le s\le \ell$. This completes the proof of Lemma \ref{ed2}.
\end{proof}

\vskip .1in
\subsection{The dissipation of the magnetic field} \label{dis}
The MHD system concerned in this paper involves
no magnetic diffusion. We need to exploit the hidden dissipation due to the background magnetic field.
The goal of this subsection is to establish the upper bound stated in the following lemma.

\begin{lemma} \label{gg}
Assume $\delta>0$ in (\ref{ping23}) is taken to be sufficiently small, then, for any $t\in [0, T]$,
\begin{align}\label{han3}
	&\norm{{\n}\cdot\nabla \b(t)}{H^{r+3}}^2-\frac{d}{dt}\sum_{0\le s\le {r+3}}\int_{\T^3}\la^{s}\p\u\cdot\la^{s}({\n}\cdot\nabla \b)\,dx\nn\\
&\quad\le C\norm{ \u(t)}{H^{r+5}}^2+C\delta^2\norm{a + \n\cdot \b}{H^{r+4}}^2.
\end{align}
\end{lemma}

\begin{proof}
Define
$$
\mathbb{Q} = \nabla\Delta^{-1}\div\quad \mbox{and}\quad \p=\mathbb{I}-\mathbb{Q},
$$
where $\mathbb P$ is the standard Leray projector operator. It is easy to check that
\begin{align*}
\p\left({\n}\cdot\nabla \b\right)=&{\n}\cdot\nabla \b, \qquad\p\left(\nabla ({\n}\cdot\b)\right)=0,\nn\\
\q\left({\n}\cdot\nabla \b\right)=&0, \qquad\q\left(\nabla ({\n}\cdot\b)\right)=\nabla ({\n}\cdot\b).
\end{align*}
Applying operator $\p$ to the second equation of \eqref{m3} gives
\begin{align}\label{ping17}
\partial_t \p\u-\mu\Delta \p\u ={\n}\cdot\nabla \b+\p f_4,
\end{align}
where $f_4$ is slightly different from $f_2$, namely
\begin{align*}
f_4\stackrel{\mathrm{def}}{=}&-\u\cdot\nabla \u+\b\cdot\nabla\b+\b\nabla\b+k(a)\nabla a\nn\\
&-I(a)(\mu\Delta \u +  (\lambda+\mu)\nabla \div \u)-I(a)({\n}\cdot\nabla \b+\b\cdot\nabla\b-{\n}\nabla \b+\b\nabla\b).
\end{align*}
Applying ${\la^s} (0\le s\le r+3) $ to  \eqref{ping17}, and taking the $L^2$-inner product with ${\la^s}({\n}\cdot\nabla \b)$, we obtain
\begin{align}\label{ping18}
\norm{\la^s({\n}\cdot\nabla \b)}{L^2}^2
=&\int_{\T^3}\la^s\partial_t \p\u\cdot\la^s({\n}\cdot\nabla \b)\,dx\nn\\
&-\mu\int_{\T^3}\la^s \Delta\p\u\cdot\la^s({\n}\cdot\nabla \b)\,dx-\int_{\T^3}\la^s(\p f_4)\cdot\la^s({\n}\cdot\nabla \b)\,dx.
\end{align}
By H\"older's inequality,
\begin{align*}
\int_{\T^3}\la^s \Delta\p\u\cdot\la^s({\n}\cdot\nabla \b)\,dx
\le&\frac18\norm{\la^s({\n}\cdot\nabla \b)}{L^2}^2+C\norm{\la^{s+2}\u}{L^2}^2, \nn\\
\int_{\T^3}\la^s(\p f_4)\cdot\la^s({\n}\cdot\nabla \b)\,dx
\le&\frac18\norm{\la^s({\n}\cdot\nabla \b)}{L^2}^2+C\norm{\la^s  f_4}{L^{2}}^2.
\end{align*}
Next we shift the time derivative in the first term on the
right hand side of \eqref{ping18} and use the third equation in \eqref{m3} to get
\begin{align}\label{ping21-1}
&\int_{\T^3}\la^s\partial_t \p\u\cdot\la^s({\n}\cdot\nabla \b)\,dx\nn\\
&\quad =\frac{d}{dt}\int_{\T^3}\la^s\p\u\cdot\la^s({\n}\cdot\nabla \b)\,dx-\int_{\T^3}\la^s\p\u\cdot\la^s({\n}\cdot\nabla \partial_t\b)\,dx\nn\\
%=&\frac{d}{dt}\int_{\T^3}\la^s\p\u\cdot\la^s({\n}\cdot\nabla \b)\,dx-\int_{\T^3}\la^s\p\u\cdot{\n}\cdot\nabla\la^s \partial_t\b\,dx\nn\\
&\quad=\frac{d}{dt}\int_{\T^3}\la^s\p\u\cdot\la^s({\n}\cdot\nabla \b)\,dx+\int_{\T^3}\la^s({\n}\cdot\nabla\p\u)\cdot\la^s \partial_t\b\,dx\nn\\
&\quad=\frac{d}{dt}\int_{\T^3}\la^s\p\u\cdot\la^s({\n}\cdot\nabla \b)\,dx+\int_{\T^3}\la^s({\n}\cdot\nabla \p\u)\cdot\la^s({\n}\cdot\nabla \u)\,dx\nn\\
&\qquad-\int_{\T^3}\la^s({\n}\cdot\nabla \p\u)\cdot\la^s(\n\div\u)\,dx+\int_{\T^3}\la^s({\n}\cdot\nabla \p\u)\cdot\la^sf_3\,dx.
\end{align}
The last three terms in \eqref{ping21-1} can be bounded by
\begin{align*}
\int_{\T^3}\la^s({\n}\cdot\nabla \p\u)\cdot\la^s({\n}\cdot\nabla \u)\,dx-\int_{\T^3}\la^s({\n}\cdot\nabla \p\u)\cdot\la^s(\n\div\u)\,dx\le C\norm{\la^{s+1}\u}{L^2}^2,
\end{align*}
and
\begin{align*}
\int_{\T^3}\la^s({\n}\cdot\nabla \p\u)\cdot\la^sf_3\,dx\le& C(\norm{ \la^s f_3}{L^{2}}^2+\norm{\la^{s+1}\u}{L^2}^2).
\end{align*}
Collecting the estimates above, we can infer from \eqref{ping18} that
\begin{align}\label{add1}
&\norm{\la^s({\n}\cdot\nabla \b)}{L^2}^2
-\frac{d}{dt}\sum_{0\le s\le {r+3}}\int_{\T^3}\la^{s}\p\u\cdot\la^{s}({\n}\cdot\nabla \b)\,dx\nn\\
&\quad\le C\big(\norm{\la^{s+2}\u}{L^2}^2+\norm{\la^{s}  f_3}{L^{2}}^2+\norm{\la^{s}  f_4}{L^{2}}^2\big).
\end{align}
In the following, we deal with the terms in $f_3, f_4,$ respectively. At first, by Lemma \ref{daishu},
\begin{align}\label{haij1}
\norm{\la^{s}  (\u\cdot\nabla \b)}{L^{2}}^2
\le&C\big(\|\u\|_{L^{\infty}}^2\|\nabla \b\|_{H^{s}}^2
        +\|\u\|_{H^{s}}^2\|\nabla \b\|_{L^{\infty}}^2\big)\nn\\
        \le&C\big(\|\u\|_{H^{3}}^2\| \b\|_{H^{N}}^2
        +\|\u\|_{H^{N}}^2\|\b\|_{H^{3}}^2\big)\nn\\
     \le&  C\delta^2 (\norm{\u}{H^3}^2+\norm{\b}{H^3}^2).
\end{align}
Similarly,
\begin{align*}%\label{haij2}
\norm{\la^{s}  (\b\cdot\nabla \u-\b\div\u)}{L^{2}}^2
     \le&  C\delta^2 (\norm{\u}{H^3}^2+\norm{\b}{H^3}^2).
\end{align*}
Moreover,
by Lemma \ref{diu},
\begin{align*}
\norm{ \b}{H^{3}}^2\le& C\norm{{\n}\cdot\nabla \b}{H^{r+3}}^2,
\end{align*}
from which we get
\begin{align}\label{haij3}
\norm{\la^{s}  f_3}{L^{2}}^2\le&  C\delta^2 \norm{\u}{H^3}^2+C\delta^2\norm{{\n}\cdot\nabla \b}{H^{r+3}}^2.
\end{align}
We now deal with the terms in $f_4$.
The term $\norm{\la^{s}  (\u\cdot\nabla \u)}{L^{2}}^2$ can be bounded as in \eqref{haij1},
\begin{align}\label{haij4}
\norm{\la^{s}  (\u\cdot\nabla \u)}{L^{2}}^2
%\le&C\big(\|\u\|_{L^{\infty}}^2\|\nabla \u\|_{H^{s}}^2
   %     +\|u\|_{H^{s}}^2\|\nabla \u\|_{L^{\infty}}^2\big)\nn\\
     \le&  C\delta^2 \norm{\la^{s+1}\u}{L^2}^2.
\end{align}
Thanks to Lemma \ref{daishu} again,
\begin{align}\label{haij7}
\norm{\la^{s}  (\b\cdot\nabla \b+\b\nabla\b)}{L^{2}}^2
\le& C(\norm{\b}{L^\infty}^2\norm{\nabla \b}{H^s}^2+\norm{\nabla \b}{L^\infty}^2\norm{ \b}{H^s}^2 )\nn\\
\le& C(\norm{\b}{H^2}^2\norm{ \b}{H^{s+1}}^2+\norm{\nabla \b}{H^2}^2\norm{ \b}{H^s}^2 )\nn\\
\le&C\norm{ \b}{H^{s+1}}^2\norm{\b}{H^3}^2
\nn\\
\le& C\delta^2\norm{{\n}\cdot\nabla \b}{H^{r+3}}^2.
\end{align}
The term $\p(k(a)\nabla a )=0$ since $k(a)\nabla a$ can be written as a gradient.  By Lemma \ref{daishu},
\begin{align}\label{haij8}
\norm{\la^{s}  (I(a)\Delta \u )}{L^{2}}^2
\le& C(\norm{I(a)}{L^\infty}^2\norm{\Delta \u}{H^s}^2+\norm{\Delta \u}{L^\infty}^2\norm{ I(a)}{H^s}^2 )\nn\\
\le& C(\norm{a}{H^2}^2\norm{ \u}{H^{s+2}}^2+\norm{\u}{H^4}^2\norm{ a}{H^s}^2 )\nn\\
\le&C\delta^2\norm{ \u}{H^{s+2}}^2+C\delta^2\norm{\u}{H^4}^2.
\end{align}
The term $I(a)\nabla \div \u$ can be  dealt with similarly. The last  term in $f_4$ can be bounded as in \eqref{ed13}, \eqref{ed14} and (\ref{haij7}) to get

\begin{align}\label{}
\norm{\la^{s}  (I(a){\n}\nabla \b )}{L^{2}}^2
\le& C(\norm{I(a)}{L^\infty}^2\norm{{\n}\nabla \b}{H^s}^2+\norm{{\n}\nabla \b}{L^\infty}^2\norm{ I(a)}{H^s}^2 )\nn\\
\le& C(\norm{a}{H^3}^2\norm{ {\n}\nabla \b}{H^{s}}^2+\norm{\b}{H^3}^2\norm{ a}{H^s}^2 )\nn\\
\le& C(\norm{a + \n\cdot \b-\n\cdot\b}{H^3}^2\norm{ \b}{H^{N}}^2+\norm{\n\cdot\nabla\b}{H^{r+3}}^2\norm{ a}{H^N}^2 )\nn\\
\le& C((\norm{a + \n\cdot \b}{H^3}^2+\norm{\b}{H^3}^2)\norm{ \b}{H^{N}}^2+\norm{\n\cdot\nabla\b}{H^{r+3}}^2\norm{ a}{H^N}^2 )\nn\\
\le& C\delta^2\norm{a + \n\cdot \b}{H^{r+4}}^2+C\delta^2\norm{{\n}\cdot\nabla \b}{H^{r+3}}^2,
\end{align}
and
\begin{align}\label{haij9}
&\norm{\la^{s}  (I(a)({\n}\cdot\nabla \b)}{L^{2}}^2+\norm{\la^{s}  (I(a)({\b}\cdot\nabla \b-{\b}\nabla \b)}{L^{2}}^2
 \le C\delta^2\|{\n}\cdot\nabla \b\|_{H^{r+3}}^2.
\end{align}
Combining \eqref{haij4}, \eqref{haij7}, \eqref{haij8} and \eqref{haij9} gives
\begin{align}\label{ahaij}
\norm{\la^{s}  f_4}{L^{2}}^2\le& C\delta^2\norm{ \u}{H^{r+5}}^2+ C\delta^2 \|{\n}\cdot\nabla \b\|_{H^{r+3}}^2+C\delta^2\norm{a + \n\cdot \b}{H^{r+4}}^2.
\end{align}
Inserting \eqref{haij3} and \eqref{ahaij} in \eqref{add1} and taking $\delta $ small enough, we obtain (\ref{han3}). This proves Lemma \ref{gg}.
\end{proof}

\vskip .1in
\subsection{The dissipation of the combined quantity $a + \n\cdot \b$} \label{dis2}

The equations of $a$ and $\b$ in (\ref{m3}) do not contain any dissipative or damping terms.
But we do need these stabilizing effects in order to prove the desired stability results.
This subsection explores the structure of (\ref{m3}) and discovers that the equation of the combined quantity
$$
a + \n\cdot \b
$$
and the equation of the gradient part $\mathbb Q \u$ of $\u$ form a system with smoothing and stabilizing  effects. Combining the bound for $a + \n\cdot \b$ and $\n\cdot \nabla \b$ allows us to control $a$.

\vskip .1in
We first derive the equation of
\begin{align*}
	{d}\stackrel{\mathrm{def}}{=} a + {{\n}\cdot\b}\quad\mbox{and}\quad
	\G\stackrel{\mathrm{def}}{=}\q\u-\frac1\nu\Delta^{-1}\nabla {d}.
\end{align*}
It follows from the third equation in \eqref{m3} that
\begin{align*}
%&a_t+ \div\u  =f_1,\\
%&\partial_t \q\u-\Delta \q\u+\nabla a =-\nabla ({\n}\cdot\b)+\q f_2,\nn\\
&\partial_t ({\n}\cdot\b)={\n}\cdot\nabla \u\cdot{\n}-\n\cdot{\n}\div\u+f_3\cdot{\n}
\end{align*}
which, together with the equation of $a$ in \eqref{m3},  gives
\begin{align}\label{han5}
&\partial_t (a+{\n}\cdot\b)={\n}\cdot\nabla \u\cdot{\n}-(|\n|^2+1)\div\u+f_1+f_3\cdot{\n}.
\end{align}
Applying the operator $\q$ to the velocity equation in \eqref{m3} yields
\begin{align}\label{han6}
&\partial_t \q\u-\nu\Delta \q\u+\nabla a+\nabla ({\n}\cdot\b)=\q f_4,
\end{align}
with $\nu\stackrel{\mathrm{def}}{=}\lambda+2\mu$. By the definition of $\mathbb{Q} = \nabla\Delta^{-1}\div$, we note that
$$
\div \u = \div \mathbb{Q}\u = \div\G + \frac1\nu d.
$$
(\ref{han5}) and (\ref{han6}) then yield
\begin{eqnarray}\label{han8}
\left\{\begin{aligned}
&\partial_t d +\frac1\nu(|\n|^2+1){d}+ (|\n|^2+1)\div \G  ={\n}\cdot\nabla \u\cdot{\n}+f_1+f_3\cdot{\n},\\
&\partial_t \G-\nu\Delta \G =\frac1\nu(|\n|^2+1)\q\u
-\frac1\nu\Delta^{-1}\nabla ({\n}\cdot\nabla \u\cdot{\n})+\q f_4-\frac1\nu\Delta^{-1}\nabla (f_1+f_3\cdot{\n}).
\end{aligned}\right.
\end{eqnarray}
On the one hand, for any $m\ge 0$, applying ${\la^m}$  to  the first equation in \eqref{han8}, and multiplying it by ${\la^m}{d}$ lead to
\begin{align*}
&\frac12\frac{d}{dt}\norm{\la^m {d}}{L^{2}}^2+\frac1\nu(|\n|^2+1)\norm{\la^m {d}}{L^{2}}^2=\int_{\T^3}\la^m({\n}\cdot\nabla \u\cdot{\n})\cdot\la^m {d}\,dx\nn\\
&\qquad-(|\n|^2+1)\int_{\T^3}\la^m\div \G\cdot\la^m {d}\,dx+\int_{\T^3}\la^m (f_1+f_3\cdot{\n})\cdot\la^m {d}\,dx\nn\\
&\quad\le C(\norm{\la^m\nabla\u}{L^{2}}\norm{\la^m {d}}{L^{2}}+\norm{\la^m\div \G}{L^{2}}\norm{\la^m {d}}{L^{2}}+\int_{\T^3}\la^m(f_1+f_3\cdot{\n})\cdot\la^m {d}\,dx)\nn\\
&\quad\le\frac{1}{8\nu}\norm{\la^m {d}}{L^{2}}^2+C\big(\norm{\la^{m+1} {\u}}{L^{2}}^2+\norm{\la^{m+1} \G}{L^{2}}^2+\norm{\la^{m}f_1}{L^{2}}^2+\norm{\la^{m}f_3}{L^{2}}^2\big)
\end{align*}
from which we have
\begin{align}\label{han9+1}
&\frac12\frac{d}{dt}\norm{\la^m {d}}{L^{2}}^2+\frac{1}{2\nu}(|\n|^2+1)\norm{\la^m {d}}{L^{2}}^2\nn\\
&\quad\le C\big(\norm{\la^{m+1} {\u}}{L^{2}}^2+\norm{\la^{m+1} \G}{L^{2}}^2+\norm{\la^{m}f_1}{L^{2}}^2+\norm{\la^{m}f_3}{L^{2}}^2\big).
\end{align}
On the other hand, for the second equation in \eqref{han8} and for any $m\ge0$, there holds similarly that
\begin{align}\label{jiajia1}
	&\frac12\frac{d}{dt}\norm{\la^{m} \G}{L^{2}}^2+\nu\norm{\la^{m+1} \G}{L^{2}}^2\nn\\
	&\quad=\frac1\nu(|\n|^2+1)\int_{\T^3}\la^m\q\u\cdot\la^m \G\,dx-\frac1\nu\int_{\T^3}\la^m (\Delta^{-1}\nabla ({\n}\cdot\nabla \u\cdot{\n}))\cdot\la^m \G\,dx\nn\\
	&\qquad+\int_{\T^3}\la^m\q f_4\cdot\la^m \G\,dx+\frac1\nu\int_{\T^3}\la^m\Delta^{-1}\nabla(f_1+f_3\cdot{\n})\cdot\la^m G\,dx.
\end{align}
For $m=0,$ we get  by the Young inequality and the Poincar\'e inequality that
\begin{align}\label{jiajia2}
&\frac12\frac{d}{dt}\norm{ \G}{L^{2}}^2+\nu\norm{\nabla \G}{L^{2}}^2\nn\\
&\quad=\frac1\nu(|\n|^2+1)\int_{\T^3}\q\u\cdot \G\,dx-\frac1\nu\int_{\T^3} (\Delta^{-1}\nabla ({\n}\cdot\nabla \u\cdot{\n}))\cdot \G\,dx\nn\\
&\qquad+\int_{\T^3}\q f_4\cdot \G\,dx+\frac1\nu\int_{\T^3}\Delta^{-1}\nabla(f_1+f_3\cdot{\n})\cdot G\,dx\nn\\
&\quad\le C(\norm{\u}{L^{2}}+\norm{f_4}{L^{2}}
+\norm{\Delta^{-1}\nabla(f_1+f_3\cdot{\n})}{L^{2}})\norm{ \G}{L^{2}}\nn\\
&\quad\le
\frac{\nu}{2}\norm{\nabla \G}{L^{2}}^2+
 C\big(\norm{\u}{H^{1}}^2+\norm{(f_1,f_3)}{H^{-1}}^2
+\norm{f_4}{L^2}^2\big).
\end{align}
For $1\le m\le N,$ we get  by the integration by parts and the Young inequality that
\begin{align*}
	&\frac12\frac{d}{dt}\norm{\la^{m} \G}{L^{2}}^2+\nu\norm{\la^{m+1} \G}{L^{2}}^2\nn\\
	&\quad\le C\norm{\la^{m-1}\u}{L^{2}}\norm{\la^{m +1} \G}{L^{2}}+C\norm{\la^{m-1}f_4}{L^{2}}\norm{\la^{m+1} \G}{L^{2}}\nn\\
	&\qquad +C\norm{\la^{m-2}(f_1+f_3\cdot{\n})}{L^{2}}\norm{\la^{m +1} G}{L^{2}}\nn\\
	&\quad\le\frac\nu4\norm{\la^{m+1}\G}{L^{2}}^2
	+C\norm{\la^{m-1}\u}{L^{2}}^2\nn\\
	&\qquad+C\norm{\la^{m-2}f_1}{L^{2}}^2+C\norm{\la^{m-2}f_3}{L^{2}}^2+C\norm{\la^{m-1}f_4}{L^{2}}^2\nn\\
	&\quad\le\frac\nu4\norm{\la^{m+1}\G}{L^{2}}^2
	+C\norm{\la^{m+1}\u}{L^{2}}^2\nn\\
	&\qquad +C\norm{\la^{m-2}f_1}{L^{2}}^2+C\norm{\la^{m-2}f_3}{L^{2}}^2+C\norm{\la^{m-1}f_4}{L^{2}}^2
\end{align*}
from which and \eqref{jiajia2}, we have for any $0\le m\le N$ that
\begin{align}\label{han10+1}
	&\frac12\frac{d}{dt}\norm{\la^{m} \G}{L^{2}}^2+\frac\nu2\norm{\la^{m+1} \G}{L^{2}}^2\nn\\
	&\quad\le C(\norm{\nabla \u}{H^{m}}^2+\norm{f_1}{H^{m}}^2
	+\norm{f_3}{H^{m}}^2+\norm{f_4}{H^{{m-1}}}^2).
\end{align}
%Noticing the fact
%\begin{align}\label{}
%\norm{\la^{m}\u}{L^{2}}^2\le C\norm{\nabla\la^{m}\u}{L^{2}}^2,\quad\norm{\la^{m-1}f_1}{L^{2}}^2\le C\norm{\la^{m}f_1}{L^{2}}^2. \qquad {\color{red}\hbox{Here we need to check!}}
%\end{align}
Multiplying \eqref{han10+1} by a suitable large constant and
then adding to \eqref{han9+1}, we get
\begin{align}\label{han11}
	&\frac{d}{dt}\big(\norm{\la^{m} d}{L^{2}}^2+\norm{\la^{m} \G}{L^{2}}^2\big)+\frac1\nu\norm{\la^{m} d}{L^{2}}^2+\nu\norm{\la^{m+1} \G}{L^{2}}^2\nn\\
	&\quad\le C\big(\norm{\nabla \u}{H^{m}}^2+\norm{f_1}{H^{m}}^2
	+\norm{f_3}{H^{m}}^2+\norm{f_4}{H^{{m-1}}}^2\big).
\end{align}
Summing up $s$ from $0$ to $m$ in \eqref{ed4} gives
\begin{align}\label{han12}
	&\frac{d}{dt}\norm{(a,\u,\b)}{H^{{m}}}^2+\mu\norm{\nabla\u}{H^{{m}}}^2
	+(\lambda+\mu)\norm{\div\u}{H^{{m}}}^2\\
	&\quad\le C\Big|\sum_{{s}=0}^{m}\int_{\T^3}\la^{{s}} f_1\cdot\la^{{s}} a\,dx\Big|+C\Big|\sum_{{s}=0}^{m}\int_{\T^3}\la^{{s}} f_4\cdot\la^{{s}} \u\,dx\Big|+C\Big|\sum_{{s}=0}^{m}\int_{\T^3}\la^{{s}} f_3\cdot\la^{{s}} \b\,dx\Big|.\nn
\end{align}
Multiplying \eqref{han12} by a suitable large constant and
then adding to  \eqref{han11} lead to
\begin{align}\label{han13456}
	&\frac{d}{dt}\norm{(a,\u,\b,d,\G)}{H^{{m}}}^2+\frac{1}{\nu}\norm{{d}}{H^{m}}^2+\mu\norm{\nabla\u}{H^{{m}}}^2+(\lambda+\mu)\norm{\div\u}{H^{{m}}}^2
	+\nu\norm{\nabla \G}{H^{m}}^2\nn\\
	&\quad\le C\big(\norm{f_1}{H^{m}}^2
	+\norm{f_3}{H^{m}}^2+\norm{f_4}{H^{{m-1}}}^2\big)+C\Big|\sum_{{s}=0}^{m}\int_{\T^3}\la^{{s}} f_1\cdot\la^{{s}} a\,dx\Big|\nn\\
	&\qquad+C\Big|\sum_{{s}=0}^{m}\int_{\T^3}\la^{{s}} f_3\cdot\la^{{s}} \b\,dx\Big|+C\Big|\sum_{{s}=0}^{m}\int_{\T^3}\la^{{s}} f_4\cdot\la^{{s}} \u\,dx\Big|.
\end{align}
Thanks to the Young inequality and the Poincar\'e inequality, for ${s}=0$, the last term in \eqref{han13456} can be bounded as
\begin{align*}%\label{han109}
\Big|\int_{\T^3} f_4\cdot \u\,dx\Big|
\le& \frac\mu8\norm{ \u}{L^{{2}}}^2+C\norm{f_4}{L^{2}}^2
\le \frac\mu8\norm{\nabla \u}{L^{{2}}}^2+C\norm{f_4}{L^{{2}}}^2.
\end{align*}
Similarly, for $1\le {s}\le m$, we have
\begin{align*}%\label{han109}
\Big|\sum_{{s}=1}^{m}\int_{\T^3}\la^{{s}} f_4\cdot\la^{{s}} \u\,dx\Big|
\le&  \frac\mu8\norm{\nabla \u}{H^{{m}}}^2+C\norm{f_2}{H^{{m-1}}}^2.
\end{align*}
Inserting the above two inequalities  into \eqref{han13456} gives
\begin{align}\label{han13}
	&\frac{d}{dt}\norm{(a,\u,\b,d,\G)}{H^{{m}}}^2+\frac{1}{\nu}\norm{{d}}{H^{m}}^2+\mu\norm{\nabla\u}{H^{{m}}}^2\nn\\ &\qquad +(\lambda+\mu)\norm{\div\u}{H^{{m}}}^2
	+\nu\norm{\nabla \G}{H^{m}}^2\nn\\
	&\quad\le C\big(\norm{f_1}{H^{m}}^2
	+\norm{f_3}{H^{m}}^2+\norm{f_4}{H^{{m-1}}}^2+\norm{f_4}{L^{{2}}}^2\big)\nn\\
	&\qquad +C\Big|\sum_{{s}=0}^{m}\int_{\T^3}\la^{{s}} f_1\cdot\la^{{s}} a\,dx\Big|+C\Big|\sum_{{s}=0}^{m}\int_{\T^3}\la^{{s}} f_3\cdot\la^{{s}} \b\,dx\Big|.
\end{align}
Taking $m=r+4$ in (\ref{han13}) implies that
\begin{align}\label{han15}
	&\frac12\frac{d}{dt}\norm{(a,\u,\b,d,\G)}{H^{{r+4}}}^2
	+\frac{1}{\nu}\norm{{d}}{H^{r+4}}^2+\mu\norm{\nabla\u}{H^{{r+4}}}^2\nn\\
	&\qquad +(\lambda+\mu)\norm{\div\u}{H^{{r+4}}}^2
	+\nu\norm{\nabla \G}{H^{r+4}}^2\nn\\
	&\quad\le C\big(\norm{f_1}{H^{r+4}}^2
	+\norm{f_3}{H^{r+4}}^2+\norm{f_4}{H^{{r+3}}}^2\big)\nn\\
	&\qquad  +C\Big|\sum_{{s}=0}^{r+4}\int_{\T^3}\la^{{s}} f_1\cdot\la^{{s}} a\,dx\Big|+C\Big|\sum_{{s}=0}^{r+4}\int_{\T^3}\la^{{s}} f_3\cdot\la^{{s}} \b\,dx\Big|.
\end{align}
Multiplying \eqref{han15} by a suitable large constant $\gamma$ and
adding to \eqref{han3} give rise to
\begin{align}\label{han169}	&\frac{d}{dt}\Bigg\{\gamma\Big(\norm{a}{H^{{r+4}}}^2+\norm{(d,\u,\b,\G)}{H^{{r+4}}}^2\Big)-\sum_{0\le s\le {r+3}}\int_{\T^3}\la^{s}\p\u\cdot\la^{s}({\n}\cdot\nabla \b)\,dx\Bigg\}\nn\\
	&\qquad	+\norm{{\n}\cdot\nabla \b}{H^{r+3}}^2 +\gamma\Big(\frac{1}{\nu}\norm{{d}}{H^{r+4}}^2+\mu\norm{\nabla\u}{H^{{r+4}}}^2+(\lambda+\mu)\norm{\div\u}{H^{{r+4}}}^2
	+\nu\norm{\nabla \G}{H^{r+4}}^2\Big)\nn\\
	&\quad\le C\big(\norm{f_1}{H^{r+4}}^2
	+\norm{f_3}{H^{r+4}}^2+\norm{f_4}{H^{{r+3}}}^2\big)\nn\\
	&\qquad  +C\Big|\sum_{{s}=0}^{r+4}\int_{\T^3}\la^{{s}} f_1\cdot\la^{{s}} a\,dx\Big|+C\Big|\sum_{{s}=0}^{r+4}\int_{\T^3}\la^{{s}} f_3\cdot\la^{{s}} \b\,dx\Big|.
\end{align}
We now bound the terms on the right hand side of \eqref{han169}.
By Lemma \ref{daishu},
\begin{align}\label{ming3}
	\norm{f_1}{H^{r+4}}^2
	\le & C(\norm{\u}{H^{{r+4}}}^2\norm{\nabla a}{H^{{r+4}}}^2+\norm{a}{H^{{r+4}}}^2\norm{\nabla\u}{H^{{r+4}}}^2)\nn\\
	\le &C\norm{\nabla\u}{H^{{r+4}}}^2\norm{a}{H^{{N}}}^2\nn\\
	\le &C\delta^2\norm{\nabla\u}{H^{{r+4}}}^2.
\end{align}
Similarly,
\begin{align}\label{ming4}
	\norm{f_3}{H^{r+4}}^2
	\le & C(\norm{\u}{H^{{r+4}}}^2\norm{\nabla\b}{H^{{r+4}}}^2+\norm{\b}{H^{{r+4}}}^2\norm{\nabla\u}{H^{{r+4}}}^2)\nn\\
	\le &C(\norm{\u}{H^{{r+4}}}^2\norm{\b}{H^{{N}}}^2+\norm{\b}{H^{{N}}}^2\norm{\nabla\u}{H^{{r+4}}}^2)\nn\\
	\le &C\norm{\nabla\u}{H^{{r+4}}}^2\norm{\b}{H^{{N}}}^2\nn\\
	\le &C\delta^2\norm{\nabla\u}{H^{{r+4}}}^2.
\end{align}
We turn to the term $|\sum_{{s}=0}^{r+4}\int_{\T^3}\la^{{s}} f_1\cdot\la^{{s}} a\,dx|$. By
Lemma \ref{jiaohuanzi},
\begin{align}\label{tain1}
&\sum_{{s}=0}^{r+4}\int_{\T^3}(\la^{{s}} (\u\cdot\nabla{a})-\u\cdot\nabla\la^{{s}}{a})\cdot\la^{{s}} {a}\,dx+\sum_{{s}=0}^{r+4}\int_{\T^3}\u\cdot\nabla\la^{{s}}{a}\cdot\la^{{s}} {a}\,dx\nn\\
&\quad\le C\sum_{{s}=0}^{r+4}(\norm{\nabla \u}{L^\infty}\norm{\la^{{s}}{a}}{L^2}+\norm{\la^{{s}} \u}{L^2}\norm{\nabla {a}}{L^\infty})\norm{{a}}{H^{{r+4}}}+C\norm{\nabla \u}{L^\infty}\norm{{a}}{H^{{r+4}}}^2\nn\\
&\quad\le C\norm{\nabla\u}{H^{{r+4}}}\norm{{a}}{H^{{r+4}}}^2.
\end{align}
By Lemma \ref{daishu}, there holds
\begin{align}\label{tain2}
\sum_{{s}=0}^{r+4}\int_{\T^3}\la^{{s}} ({a}\div\u)\cdot\la^{{s}} {a}\,dx\le& C\norm{\nabla\u}{H^{{r+4}}}\norm{{a}}{H^{{r+4}}}^2.
\end{align}
As a result, we have
\begin{align}\label{tain3}
\Big|\sum_{{s}=0}^{r+4}\int_{\T^3}\la^{{s}} f_1\cdot\la^{{s}} a\,dx\Big|
\le& C\norm{\nabla\u}{H^{{r+4}}}\norm{{a}}{H^{{r+4}}}^2\nn\\
\le&\frac{\mu}{8}\norm{\nabla\u}{H^{{r+4}}}^2+C\norm{{d}}{H^{{r+4}}}^4+C\norm{{\b}}{H^{{r+4}}}^4\nn\\
\le&\frac{\mu}{8}\norm{\nabla\u}{H^{{r+4}}}^2+C\delta^2\norm{{d}}{H^{{r+4}}}^2+C\norm{{\b}}{H^{{r+4}}}^4.
\end{align}
Similarly, the last term in \eqref{han169} can be bounded as
\begin{align}\label{ming736}
\Big|\sum_{{s}=0}^{r+4}\int_{\T^3}\la^{{s}} f_3\cdot\la^{{s}} \b\,dx\Big|
%\le& C\norm{\nabla\u}{H^{{r+4}}}\norm{\b}{H^{{r+4}}}^2\nn\\
\le&\frac{\mu}{8}\norm{\nabla\u}{H^{{r+4}}}^2+C\norm{\b}{H^{{r+4}}}^4.
\end{align}
For any ${N}\ge 2r+5$, from Lemma \ref{diu}, we have
\begin{align*}
	\norm{ \b}{H^{3}}\le& C\norm{{\n}\cdot\nabla \b}{H^{r+3}},\quad\hbox{and}\quad
	\norm{ \b}{H^{r+4}}^2\le C\norm{ \b}{H^{3}}\norm{ \b}{H^{{N}}}\le C\delta\norm{{\n}\cdot\nabla \b}{H^{r+3}},
\end{align*}
which gives
\begin{align*}
	\norm{ \b}{H^{r+4}}^4\le C\delta^2\norm{{\n}\cdot\nabla \b}{H^{r+3}}^2.
\end{align*}
As a result, we get
\begin{align}\label{ming7}
	&\Big|\sum_{{s}=0}^{r+4}\int_{\T^3}\la^{{s}} f_1\cdot\la^{{s}} a\,dx\Big|+C\Big|\sum_{{s}=0}^{r+4}\int_{\T^3}\la^{{s}} f_3\cdot\la^{{s}} \b\,dx\Big|\nn\\
	&\quad\le\frac{\mu}{8}\norm{\nabla\u}{H^{{r+4}}}^2+C\delta^2(\norm{{\n}\cdot\nabla \b}{H^{r+3}}^2+\norm{{d}}{H^{{r+4}}}^2).
\end{align}
Finally we estimate $\norm{f_4}{H^{{r+3}}}^2$ and start with the first
term $\norm{\u\cdot\nabla \u}{H^{{r+3}}}^2$. By Lemma \ref{daishu},
\begin{align*}
	\norm{\u\cdot\nabla \u}{H^{{r+3}}}^2
	\le&C\norm{\u}{H^{{r+4}}}^2\norm{\nabla\u}{H^{{r+4}}}^2\nn\\
	\le&C\norm{\u}{H^{N}}^2\norm{\nabla\u}{H^{{r+4}}}^2\nn\\
	\le&C\delta^2\norm{\nabla\u}{H^{{r+4}}}^2.
\end{align*}
Similarly, by Lemma \ref{diu},
\begin{align*}
	\norm{\b\cdot\nabla\b+\b\nabla\b}{H^{{r+3}}}^2
	\le&C\norm{ \b}{H^{3}}^2\norm{\nabla\b}{H^{r+3}}^2\nn\\
	\le& C\norm{{\n}\cdot\nabla \b}{H^{r+3}}^2\norm{\b}{H^{N}}^2\nn\\
	\le&C\delta^2\norm{{\n}\cdot\nabla \b}{H^{r+3}}^2.
\end{align*}
With the aid of Lemma \ref{daishu} again,  we can deduce that
\begin{align*}
	\norm{k(a)\nabla a}{H^{r+3}}^2
	\le& C\left(\norm{\nabla a}{H^{r+3}}^2\norm{k(a)}{L^{\infty}}^2+\norm{k(a)}{H^{r+3}}^2\norm{\nabla a}{L^{\infty}}^2\right)\nn\\
	\le& C\norm{a}{H^{3}}^2\norm{a}{H^{N}}^2\nn\\
	\le& C\norm{d-\n\cdot\b}{H^{3}}^2\norm{a}{H^{N}}^2\nn\\
	\le& C(\norm{d}{H^{3}}^2+\norm{\b}{H^{3}}^2)\norm{a}{H^{N}}^2\nn\\
	\le& C\delta^2\norm{d}{H^{r+4}}^2+C\delta^2\norm{{\n}\cdot\nabla \b}{H^{r+3}}^2
\end{align*}
and
\begin{align*}
	\norm{I(a)(\mu\Delta \u +  (\lambda+\mu)\nabla \div \u)}{H^{r+3}}^2
	\le&C\norm{a}{H^{N}}^2\norm{\Delta \u}{H^{{r+3}}}^2\nn\\
	\le&  C\delta^2\norm{\nabla\u}{H^{{r+4}}}^2.
\end{align*}
By Lemmas \ref{daishu} and \ref{fuhe},  and (\ref{eq:smalla}),
\begin{align}\label{uu6}
	\norm{I(a)({\n}\cdot\nabla \b-{\n}\nabla \b)}{H^{r+3}}^2
	\le&C
	(\|I(a)\|_{L^\infty}^2\|\nabla \b\|_{H^{r+3}}^2+\|\nabla \b\|_{L^\infty}^2\|I(a)\|_{H^{r+3}}^2)\nn\\
	\le&C(\norm{\b}{H^{N}}^2 \norm{a}{H^{3}}^2+\norm{ \b}{H^{3}}^2\norm{a}{H^{r+4}}^2)\nn\\
	\le&C\delta^2\norm{d-\n\cdot\b}{H^{3}}^2+C\norm{{\n}\cdot\nabla \b}{H^{r+3}}^2\norm{a}{H^{N}}^2\nn\\
	\le&C\delta^2(\norm{d}{H^{3}}^2+\norm{\b}{H^{3}}^2)+C\norm{{\n}\cdot\nabla \b}{H^{r+3}}^2\norm{a}{H^{N}}^2\nn\\
	\le&C\delta^2\norm{d}{H^{r+4}}^2+C\delta^2\norm{{\n}\cdot\nabla \b}{H^{r+3}}^2.
\end{align}
The last term in $\norm{f_4}{H^{r+3}}^2$ can be dealt with similarly as \eqref{uu6}. Collecting the estimates above yields
\begin{align}\label{uu7}
	\norm{f_4}{H^{r+3}}^2\le C\delta^2\norm{\nabla\u}{H^{{r+4}}}^2+C\delta^2\norm{{\n}\cdot\nabla \b}{H^{r+3}}^2
	+C\delta^2\norm{d}{H^{{r+4}}}^2.
\end{align}
Inserting \eqref{ming3}, \eqref{ming4}, \eqref{ming7} and \eqref{uu7} in \eqref{han169} and taking $\delta>0$ to be sufficiently small, we obtain
\begin{align}\label{ming11}
	&\frac{d}{dt}\Bigg\{\gamma\Big(\norm{(a,\u,\b,d,\G)}{H^{{r+4}}}^2\Big)-\sum_{0\le s\le {r+3}}\int_{\T^3}\la^{s}\p\u\cdot\la^{s}({\n}\cdot\nabla \b)\,dx\Bigg\}+\norm{{\n}\cdot\nabla \b}{H^{r+3}}^2\nn\\
	&\qquad
	+\gamma\Big(\frac{1}{\nu}\norm{{d}}{H^{r+4}}^2+\mu\norm{\nabla\u}{H^{{r+4}}}^2+(\lambda+\mu)\norm{\div\u}{H^{{r+4}}}^2
	+\nu\norm{\nabla \G}{H^{r+4}}^2\Big)
	\le 0.
\end{align}

\subsection{Completing the proof of Theorem \ref{dingli}} \label{compl}
This subsection finishes the bootstrapping argument and thus completes the  proof of Theorem \ref{dingli}.
Let $\gamma>1$. We set
\begin{align*}%\label{ming12}
	{\mathcal{E}(t)}=&\gamma\left(\norm{a}{H^{{r+4}}}^2+\norm{(d,\u,\b,\G)}{H^{{r+4}}}^2\right)-\sum_{0\le s\le {r+3}}\int_{\T^3}\la^{s}\p\u\cdot\la^{s}({\n}\cdot\nabla \b)\,dx,
\end{align*}
and
\begin{align*}%\label{ming12}
	{\mathcal{D}(t)}=&\gamma\Big(\frac{1}{\nu}\norm{{d}}{H^{r+4}}^2+\mu\norm{\nabla\u}{H^{{r+4}}}^2+(\lambda+\mu)\norm{\div\u}{H^{{r+4}}}^2
	+\nu\norm{\nabla \G}{H^{r+4}}^2\Big)\\
	&+\norm{{\n}\cdot\nabla \b}{H^{r+3}}^2.
\end{align*}
(\ref{ming11}) then becomes
\begin{align}\label{ming13}
	\frac{d}{dt}{\mathcal{E}(t)}+\frac12{\mathcal{D}(t)}\le 0.
\end{align}
Clearly, for  $\gamma>1$,
$${\mathcal{E}(t)}\ge\norm{(d,\u,\b,\G)}{H^{{r+4}}}^2.$$
For any ${N}\ge 4r+7$, we invoke the interpolation inequality
\begin{align*}%\label{ming14}
	\norm{ \b}{H^{r+4}}^2\le&\norm{ \b}{H^{3}}^{\frac32}\norm{ \b}{H^{{N}}}^{\frac12}\le C\delta^{\frac12}\norm{{\n}\cdot\nabla \b}{H^{r+3}}^{\frac32}
\end{align*}
to obtain
\begin{align*}%\label{ming15}
	{\mathcal{E}(t)}\le& C(\norm{d}{H^{r+4}}^2+\norm{(\u,\G)}{H^{r+4}}^2+\norm{\b}{H^{r+4}}^2)\nn\\
	\le& C\norm{ d}{H^{r+4}}^{\frac32}\norm{ d}{H^{{r+4}}}^{\frac12}+C\norm{(\u,\G)}{H^{3}}^{\frac32}\norm{ (\u,\G)}{H^{{N}}}^{\frac12}+C\norm{ \b}{H^{3}}^{\frac32}\norm{ \b}{H^{{N}}}^{\frac12}\nn\\
	\le& C\norm{ d}{H^{r+4}}^{\frac32}\norm{ d}{H^{{N}}}^{\frac12}+C\norm{(\u,\G)}{H^{3}}^{\frac32}\norm{ (\u,\G)}{H^{{N}}}^{\frac12}+C\norm{ \b}{H^{3}}^{\frac32}\norm{ \b}{H^{{N}}}^{\frac12}\\
	\le& C\delta^{\frac12}\norm{  d}{H^{r+4}}^{\frac32}+
	C\delta^{\frac12}\norm{\nabla (\u,\G)}{H^{r+4}}^{\frac32}+C\delta^{\frac12}\norm{{\n}\cdot\nabla \b}{H^{r+3}}^{\frac32}\nn\\
	\le& C (\mathcal{D}(t))^{\frac34}.
\end{align*}
Inserting this inequality in (\ref{ming13}) gives
\begin{align*}%\label{ming16}
	\frac{d}{dt}{\mathcal{E}(t)}+c({\mathcal{E}(t)})^{\frac43}\le 0.
\end{align*}
It then follows easily that
\begin{align}\label{ming17}
	{\mathcal{E}(t)}\le C(1+t)^{-3}.
\end{align}
Taking $\ell={N}$ in \eqref{ed3} and using Sobolev's inequalities, we have
\begin{align}\label{ming18}
	&\frac{d}{dt}\norm{(a,\u,\b)}{H^{N}}^2+\mu\norm{\nabla\u}{H^{N}}^2+(\lambda+\mu)\norm{\div\u}{H^{N}}^2
	\le CZ(t)\norm{(a,\u,\b)}{H^{N}}^2
\end{align}
with
\begin{align*}
	Z(t)\stackrel{\mathrm{def}}{=}&\norm{(d,\u,\b)}{H^3}+\norm{(d,\u,\b)}{H^3}^2+\norm{d}{H^3}^2\norm{\b}{H^3}^2+\norm{\b}{H^3}^4.
\end{align*}
Clearly, the decay upper bound in (\ref{ming17}) implies
\begin{align*}
	\int_0^tZ(\tau)\,d\tau\le C.
\end{align*}
Gr\"onwall's inequality applied to \eqref{ming18} implies
\begin{align*}
	\norm{(a,\u,\b)}{H^{N}}^2
	\le&C\norm{(a_0,\u_0,\b_0)}{H^{N}}^2
	\le C\varepsilon^2.
\end{align*}
By taking $\varepsilon$ to be sufficiently small, say $\sqrt{C}\varepsilon\le \delta/2$, we
obtain
$$
	\norm{(a,\u,\b)}{H^{N}}
	\le \frac{\delta}{2}.
$$
The bootstrapping argument then implies that the local solution
can be extended as a global one in time. Finally we prove the decay rate in (\ref{dec}).
By \eqref{ming17},
\begin{align}\label{ming21}
	\norm{a(t)}{H^{r+4}}+\norm{\u(t)}{H^{r+4}}+\norm{\b(t)}{H^{r+4}}\le C(1+t)^{-\frac{3}{2}}.
\end{align}
(\ref{dec}) is a consequence of (\ref{ming21}) and the interpolation inequality,
for any $r+4 \le \beta<N$,
\begin{align*}%\label{ming22}
	\norm{f(t)}{H^{\beta}}\le\norm{f(t)}{H^{r+4}}^{\frac{{N}-\beta}{{N}-r-4}}
	\norm{f(t)}{H^{N}}^{\frac{\beta-r-4}{{N}-r-4}}.
\end{align*}
This completes the proof of Theorem \ref{dingli}.$\quad\square$

\bigskip
\section*{Acknowledgments}
Part of this work was done when Xiaoping Zhai was visiting School of Mathematical
Science, Peking University in the  Summer of 2021. Zhai would like to thank Professor Zhifei Zhang for his value discussions and suggestions. Wu was partially supported by the National Science Foundation of the United
  States under DMS 2104682 and DMS 2309748.  Zhai was partially supported by the Guangdong Provincial Natural Science Foundation under grant 2022A1515011977 and the
Science and Technology Program of Shenzhen under grant 20200806104726001.

\vskip .1in
\noindent{Conflict of Interest:} The authors declare that they have no conflict of interest.

\vskip .1in
\noindent{Data availability statement:}
 Data sharing not applicable to this article as no datasets were generated or analysed during the
current study.

\end{document}